\documentclass[12pt,twoside]{amsart}
\usepackage{amssymb,amscd,amsxtra,calc}
\usepackage{amsmath}
\usepackage{xypic}
\usepackage[graph]{xy}

\title[Log del Pezzo surfaces]{Log del Pezzo surfaces with not small fractional indices}
\author{Kento Fujita} 
\date{\today}
\address{Research Institute for Mathematical Sciences, 
Kyoto University, Kyoto 606-8502 Japan}
\email{fujita@kurims.kyoto-u.ac.jp}

\newcommand{\pr}{\mathbb{P}}

\newcommand{\Z}{\mathbb{Z}}
\newcommand{\Q}{\mathbb{Q}}

\newcommand{\C}{\mathbb{C}}
\newcommand{\F}{\mathbb{F}}

\newcommand{\Pic}{\operatorname{Pic}}

\newcommand{\discrep}{\operatorname{discrep}}

\newcommand{\mult}{\operatorname{mult}}
\newcommand{\coeff}{\operatorname{coeff}}
\newcommand{\FS}{\operatorname{FS}}

\newcommand{\bi}{\bfseries\itshape}

\newcommand{\sI}{\mathcal{I}}

\newcommand{\sO}{\mathcal{O}}

\newcommand{\dm}{\mathfrak{m}}

\newtheorem{thm}{Theorem}[section]
\newtheorem{lemma}[thm]{Lemma}
\newtheorem{proposition}[thm]{Proposition}
\newtheorem{corollary}[thm]{Corollary}
\newtheorem{claim}[thm]{Claim}

\newtheorem{fact}[thm]{Fact}

\theoremstyle{definition}
\newtheorem{definition}[thm]{Definition}
\newtheorem{remark}[thm]{Remark}
\newtheorem{notation}[thm]{Notation}
\newtheorem{example}[thm]{Example}
\newtheorem*{ack}{Acknowledgments} 
\newtheorem*{nott}{Notation and terminology}

\begin{document}

\maketitle 

\begin{abstract}
For a log del Pezzo surface $S$, the fractional index $r(S)\in\Q_{>0}$ is the 
maximum of $r$ with which $-K_S$ can be written as $r$ times some Cartier divisor. 
We classify all the log del Pezzo surfaces $S$ with $r(S)\geq 1/2$, after the 
technique of Nakayama. 
\end{abstract}


\section{Introduction}\label{intro_section}

The aim of the article is to classify a certain class of normal projective surface 
with some positivity condition. A normal projective variety $V$ is called a 
\emph{log Fano variety} if $V$ is log-terminal (that is, the canonical divisor 
$K_V$ is $\Q$-Cartier and the discrepancy of $V$ is strictly bigger than $-1$) and 
the anti-canonical $-K_V$ is ample ($\Q$-Cartier). Two-dimensional log Fano varieties 
are denoted by \emph{log del Pezzo surfaces}. 
If the base field is the complex number field $\C$, 
then the notion of log-terminal surface singularities is nothing but the notion of 
quotient singularities. 
For a log Fano variety $V$, let 
\[
r(V):=\sup\{r\in\Q_{>0}\,\,|\,\,-K_V\equiv rL\, \,\text{ for some Cartier divisor }\,
L\}
\]
and we call $r(V)$ the \emph{fractional index} of $V$. 
Around the end of 1980's, Alexeev \cite{A88, A91} considered the following set: 
\[
\FS_n:=\{r(V)\,\,|\,\,V:\,n\text{-dimensional log Fano variety}\}.
\]
The set is closely related to the minimal model program. 

\begin{fact}\label{fa}
Assume that the base field is the complex number field $\C$. 
\begin{enumerate}
\renewcommand{\theenumi}{\arabic{enumi}}
\renewcommand{\labelenumi}{(\theenumi)}
\item\label{fa1}
\cite{A88}
The set $\FS_2$ satisfies the ascending chain condition. 
Moreover, the set of accumulation points of $\FS_2$ is 
\[
\{0\}\cup\left\{\frac{1}{k}\,\,\bigg|\,\,k\in\Z_{>0}\right\}.
\]
\item\label{fa2}
\cite{A91}
For any $n\geq 2$, the equality $\FS_n\cap(n-2, \infty)=\FS_2+(n-2)$
holds.
\item\label{fa3}
\cite{HMX}
The set $\FS_n$ satisfies the ascending chain condition. 
\end{enumerate}
\end{fact}

In this article, we analyze the set $\FS_2$. First, we can easily analyze the set 
$\FS_2\cap[1, \infty)$. 

\begin{proposition}[{=Proposition \ref{big_prop}}]\label{intro_big_prop}
Let $S$ be a log del Pezzo surface. 
\begin{enumerate}
\renewcommand{\theenumi}{\arabic{enumi}}
\renewcommand{\labelenumi}{(\theenumi)}
\item\label{intro_big_prop1}\cite{FJT}
If $r(S)>1$, then $S$ is isomorphic to one of 
$\pr^2$ $(r(S)=3)$, $\pr^1\times\pr^1$ $(r(S)=2)$, or 
$\pr(1,1,n)$ $(n\geq 2)$ $(r(S)=(n+2)/n)$.
\item\label{intro_big_prop2}\cite{brenton, demazure, HW}
If $r(S)=1$, then $S$ has at most du Val singularities and such $S$ has been classified. 
\end{enumerate}
\end{proposition}

We are mainly interested in the set $\FS_2\cap[1/2, 1)$. 

\begin{remark}\label{an-n-ot}
Let $S$ be a log del Pezzo surface over $\Bbbk$. 
\begin{enumerate}
\renewcommand{\theenumi}{\arabic{enumi}}
\renewcommand{\labelenumi}{(\theenumi)}
\item\label{an-n-ot1}
Assume that $-2K_S$ is Cartier. Such $S$ has been classified by 
Alexeev-Nikulin \cite{AN1, AN2, AN3} (for the case $\Bbbk=\C$) and Nakayama \cite{N} 
(for arbitrary $\Bbbk$).
\item\label{an-n-ot2}
Assume that $\Bbbk=\C$, $r(S)=2/3$ and $-3K_S$ is Cartier. 
Such $S$ has been treated by Ohashi-Taki \cite{OT} 
(see Remark \ref{OT_rmk}).
\end{enumerate}
\end{remark}

The main theorem in this article is to give a geometric classification of 
the log del Pezzo surfaces $S$ with $r(S)\in[1/2, 1)$ 
(Theorems \ref{mainthm1} and \ref{mainthm2}) over an 
algebraically closed field $\Bbbk$ of arbitrary characteristic. 
As a corollary, we can determine the 
set $\FS_2\cap[1/2, 1)$ (Corollary \ref{maincor}).

The main idea for the classification comes from Nakayama's technique used in \cite{N}. 
We consider the following three objects: 

\begin{itemize}
\item
A log del Pezzo surface $S$ and positive integers $a$, $b$ such that 
$r(S)=b/a\in[1/2, 1)$ and 
$-aK_S\sim bL_S$ for some Cartier divisor. The pair $(a, b)$ is called a 
\emph{multi-index} of $S$. 
\item
An \emph{$(a, b)$-basic pair} $(M, E_M)$ such that $M$ is a nonsingular projective 
rational surface, $E_M$ is an effective divisor on $M$ and satisfies the conditions 
in Section \ref{basic_section}. 
\item
An \emph{$(a, b)$-fundamental triplet} $(X, E_X, \Delta)$ such that 
$X$ is either $\pr^2$ or $\F_n$, $E_X$ is an effective divisor on $X$, $\Delta$ is 
a zero-dimensional subscheme of $X$ and satisfies the conditions in Section 
\ref{triplet_section}. 
\end{itemize}

The relationship among those three are as follows. From a log del Pezzo surface $S$ 
and a multi-index $(a, b)$ of $S$, we take the minimal resolution $\alpha\colon M\to S$ 
of $S$. Set $E_M:=-aK_{M/S}$. 
Then the pair $(M, E_M)$ is an $(a, b)$-basic pair. 
Conversely, from an $(a, b)$-basic pair $(M, E_M)$, some positive multiple of 
$-aK_M-E_M$ determines a birational morphism $\alpha\colon M\to S$ such that 
$S$ is a log del Pezzo surface, $(a, b)$ is a multi-index of $S$ and $\alpha$ is the 
minimal resolution of $S$. The relationship is treated in Section \ref{basic_section}. 
From an $(a, b)$-basic pair $(M, E_M)$, by taking $(-E_M)$-minimal model program, 
we get a birational morphism $\phi\colon M\to X$ such that $X=\pr^2$ or $\F_n$. Set 
$E_X:=\phi_*E_M$. There exists a zero-dimensional subscheme $\Delta\subset X$ 
satisfying the $(\nu1)$-condition such that $\phi$ is the elimination of $\Delta$ 
(see Section \ref{elim_subsec}). Then the triplet $(X, E_X, \Delta)$ is an 
$(a, b)$-fundamental triplet. Conversely, from an $(a, b)$-fundamental triplet 
$(X, E_X, \Delta)$, the pair $(M, (E_X)_M^{\Delta, a-b})$ is an 
$(a, b)$-basic pair, where $\phi\colon M\to X$ is the elimination of $\Delta$ and 
$(E_X)_M^{\Delta, a-b}=\phi^*E_X-(a-b)K_{M/X}$. 
The relationship is treated in Section \ref{triplet_section}. 

In Section \ref{classify_section}, we classify all of the (normalized) $(a, b)$-fundamental 
triplets with $1/2\leq b/a<1$ and $b>1$. The main strategy to classy the triplet 
$(X, E, \Delta)$ is to exclude the possibility such that $E$ contains singular components. 
The argument is prepared in Section \ref{singcurve_section}. The idea is simple but is 
important in order to classify the triplets. 
In Section \ref{graph_section}, we tabulate the list of the dual graphs of 
non-Gorenstein singularities on $S$ (see Tables \ref{maintable1} and \ref{maintable2}). 
This is a direct consequence of the results in Section \ref{classify_section}.

\begin{ack}
The author would like to thank Doctor Kazunori Yasutake for useful discussion 
and comments.
The author is partially supported by a JSPS Fellowship for Young Scientists. 
\end{ack}

\begin{nott}
We work in the category of algebraic (separated and of finite type) scheme over a 
fixed algebraically closed field $\Bbbk$ of arbitrary characteristic. 
A \emph{variety} means a reduced and irreducible algebraic scheme. A \emph{surface} 
means a two-dimensional variety. 

For a normal variety $X$, we say that $D$ is a 
\emph{$\Q$-divisor} (resp.\ \emph{divisor} or \emph{$\Z$-divisor}) 
if $D$ is a finite sum $D=\sum a_iD_i$ where $D_i$ are prime divisors and 
$a_i\in\Q$ (resp.\ $a_i\in\Z$). 
For a $\Q$-divisor $D=\sum a_iD_i$, the value $a_i$ is denoted by $\coeff_{D_i}D$ and 
set $\coeff D:=\{a_i\}_i$. 
A normal variety $X$ is called \emph{log-terminal} if the canonical divisor $K_X$ is 
$\Q$-Cartier and the discrepancy $\discrep(X)$ of $X$ is bigger than $-1$ 
(see \cite[\S 2.3]{KoMo}).
For a proper birational morphism $f\colon Y\to X$ between normal varieties 
such that both $K_X$ and $K_Y$ are $\Q$-Cartier, we set 
\[
K_{Y/X}:=\sum_{E_0\subset Y \,\,f\text{-exceptional}}a(E_0, X)E_0, 
\]
where $a(E_0, X)$ is the discrepancy of $E_0$ with respects to $X$ 
(see \cite[\S 2.3]{KoMo}). (We note that 
if $aK_X$ and $aK_Y$ are Cartier for $a\in\Z_{>0}$, then 
$aK_{Y/X}$ is a $\Z$-divisor.)

For a nonsingular surface $S$ and a projective curve $C$ which is a closed subvariety 
of $S$, the curve $C$ is called a \emph{$(-n)$-curve} if $C$ is a nonsingular rational 
curve and $(C^2)=-n$.

Now we determine the dual graphs for divisors on surfaces. 
Let $S$ be a nonsingular surface and let $D=\sum a_jD_j$ be an effective divisor on $S$ 
$(a_j>0)$ such that $|D|$ (the support of $D$) 
is simple normal crossing and $D_i$, $D_j$ are intersected 
at most one point. (It is sufficient in our situation.) 
The \emph{dual graph} of $D$ is defined as follows. A vertex corresponds 
to a component $D_j$. Let $v_j$ be the vertex corresponds to $D_j$. The 
$v_i$ and $v_j$ are joined by a (simple) line if and only if $D_i$, $D_j$ are intersected. 
In the dual graphs of divisors, a vertex corresponding to $(-n)$-curve is expressed 
as follows: 

\begin{table}[h]
\begin{tabular}{|c|c|c|c|}\hline
$(-1)$-curve & $(-2)$-curve & $(-3)$-curve & $(-n)$-curve \\ \hline
\begin{picture}(10,10)(0,0)
     \put(5, 3){\circle{10}}
\end{picture}
&
\begin{picture}(10,10)(0,0)
     \put(5, 3){\circle*{10}}
\end{picture}
&
\begin{picture}(10,10)(0,0)
     \put(5, 3){\circle*{6}}
     \put(5, 3){\circle{10}}
\end{picture}
&
\begin{picture}(10,10)(0,0)
     \put(5, .5){\makebox(0, 0){\textcircled{\tiny $n$}}}
\end{picture}
\\ 
\hline
\end{tabular}
\end{table}

On the other hand, an arbitrary irreducible curve is expressed by the symbol 
{\large$\oslash$} when it is not necessary a $(-n)$-curve.

Let $\F_n\to\pr^1$ be a Hirzebruch surface $\pr_{\pr^1}(\sO\oplus\sO(n))$ of 
degree $n$ with the $\pr^1$-fibration. 
A section $\sigma\subset\F_n$ with $\sigma^2=-n$ 
is called a \emph{minimal section}. If $n>0$, then such $\sigma$ is unique.
A section $\sigma_{\infty}$ with $\sigma\cap\sigma_\infty=\emptyset$ is called a
\emph{section at infinity}. For a section at infinity $\sigma_\infty$, we have 
$\sigma_\infty\sim\sigma+nl$, where $l$ is a fiber of the fibration.
\end{nott}

\section{Preliminaries}\label{prelim_section}

\subsection{Elimination of subschemes}\label{elim_subsec}

We recall the results in \cite[\S 2]{N}. 
Let $X$ be a nonsingular surface and $\Delta$ be a zero-dimensional subscheme 
of $X$. The defining ideal sheaf of $\Delta$ is denoted by $\sI_{\Delta}$.

\begin{definition}\label{nu1_dfn}
Let $P$ be a point of $\Delta$. 
\begin{enumerate}
\renewcommand{\theenumi}{\arabic{enumi}}
\renewcommand{\labelenumi}{(\theenumi)}
\item\label{nu1_dfn1}
Let 
$
\nu_P(\Delta):=\max\{\nu\in\Z_{>0} \, | \, \sI_{\Delta}\subset\dm_P^\nu\},
$
where $\dm_P$ is the maximal ideal sheaf in $\sO_X$ defining $P$.
If $\nu_P(\Delta)=1$ for any $P\in\Delta$, then we say that $\Delta$ \emph{satisfies 
the $(\nu1)$-condition}.
\item\label{nu1_dfn2}
The \emph{multiplicity} $\mult_P\Delta$ of $\Delta$ at $P$ is given by 
the length of the Artinian local ring $\sO_{\Delta, P}$.
\item\label{nu1_dfn3}
The \emph{degree} $\deg\Delta$ of $\Delta$ is 
given by $\sum_{P\in\Delta}\mult_P\Delta$.
\end{enumerate}
\end{definition}

\begin{definition}\label{mult_div}
For an effective divisor $D$ and a point $P$, we set 
\[
\mult_PD:=\max\{\nu\in\Z_{>0} \, | \, \sO_X(-D)\subset\dm_P^\nu\}.
\]
Let $\pi\colon Y\to X$ be the blowing up along $P$ and let $e$ be 
the exceptional curve. Then $\mult_PD$ is equal to $\coeff_e\pi^*D$. 
\end{definition}

\begin{definition}\label{elim_dfn}
Assume that $\Delta$ satisfies the $(\nu1)$-condition. 
Let $V\to X$ be the blowing up along $\Delta$. 
The \emph{elimination} of $\Delta$ is the birational projective morphism 
$\phi\colon M\to X$ which is defined as the composition of the minimal resolution 
$M\to V$ of $V$ and the morphism $V\to X$. 
For any divisor $E$ on $X$ and for any positive integer $s$, we set 
$E_M^{\Delta,s}:=\phi^*E-sK_{M/X}$. 
\end{definition}

\begin{proposition}[{\cite[Proposition 2.9]{N}}]\label{elim_prop}
\begin{enumerate}
\renewcommand{\theenumi}{\arabic{enumi}}
\renewcommand{\labelenumi}{(\theenumi)}
\item\label{elim_prop1}
Assume that the subscheme $\Delta$ satisfies the $(\nu1)$-condition 
and let $\phi\colon M\to X$ 
be the elimination of $\Delta$. 
Then the anti-canonical divisor $-K_M$ is $\phi$-nef. 
More precisely, for any $P\in\Delta$ with $\mult_P\Delta=k$, 
the set-theoretic inverse image $\phi^{-1}(P)$ is the 
straight chain $\sum_{j=1}^k\Gamma_j$ of 
nonsingular rational curves and the dual graph of $\phi^{-1}(P)$ is the following: 
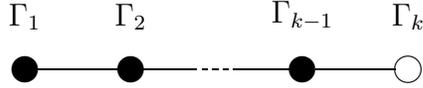
\begin{center}
    \begin{picture}(150, 35)(0, 50)
    \put(0, 60){\circle*{10}}
    \put(0, 75){\makebox(0, 0)[b]{$\Gamma_{1}$}}
    \put(5, 60){\line(1, 0){30}}
    \put(40, 60){\circle*{10}}
    \put(40, 75){\makebox(0, 0)[b]{$\Gamma_{2}$}}
    \put(45, 60){\line(1, 0){20}} 
    \put(67, 60){\line(1, 0){2}}
    \put(71, 60){\line(1, 0){2}}
    \put(75, 60){\line(1, 0){2}}    
    \put(79, 60){\line(1, 0){21}}
    \put(105, 60){\circle*{10}}
    \put(105, 75){\makebox(0, 0)[b]{$\Gamma_{k-1}$}}
    \put(110, 60){\line(1, 0){30}}
    \put(145, 60){\circle{10}}
    \put(145, 75){\makebox(0, 0)[b]{$\Gamma_{k}$}}
    \end{picture}
\end{center}
\item\label{elim_prop2}
Conversely, for a non-isomorphic proper birational morphism $\phi\colon M\to X$
between nonsingular surfaces such that $-K_M$ is $\phi$-nef, 
the morphism $\phi$ 
is the elimination of $\Delta$ which satisfies the $(\nu1)$-condition 
defined by the ideal $\sI_\Delta:=\phi_*\sO_M(-K_{M/X})$.
\end{enumerate}
\end{proposition}

Now we see some examples of the dual graphs of $E_M^{\Delta, s}$.

\begin{example}\label{Sit1}
Assume that $\Delta$ satisfies the $(\nu1)$-condition such that $|\Delta|=\{P\}$.
Let $E=mE_1$ be a non-zero effective divisor such that $E_1$ is reduced and 
nonsingular at $P\in E_1$. Let $k:=\deg\Delta$ and $l:=\mult_P(\Delta\cap E_1)$. 
By \cite[Lemma 2.17]{N}, we have 
\[
E_M^{\Delta, s}=mE_{1,M}+\sum_{i=1}^li(m-s)\Gamma_i+\sum_{i=l+1}^k(ml-si)\Gamma_i,
\]
where $E_{1,M}$ is the strict transform of $E_1$ on $M$. 
Moreover, the dual graph of $\phi^{-1}(E)$ is the following: 
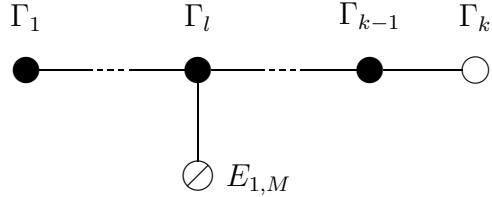
\begin{center}
    \begin{picture}(180, 90)(0, 0)
    \put(0, 60){\circle*{10}}
    \put(0, 75){\makebox(0, 0)[b]{$\Gamma_{1}$}}
    \put(5, 60){\line(1, 0){20}} 
    \put(27, 60){\line(1, 0){2}}
    \put(31, 60){\line(1, 0){2}}
    \put(35, 60){\line(1, 0){2}}    
    \put(39, 60){\line(1, 0){21}}
    \put(65, 60){\circle*{10}}
    \put(65, 75){\makebox(0, 0)[b]{$\Gamma_l$}}
    \put(70, 60){\line(1, 0){20}} 
    \put(92, 60){\line(1, 0){2}}
    \put(96, 60){\line(1, 0){2}}
    \put(100, 60){\line(1, 0){2}}    
    \put(104, 60){\line(1, 0){21}}
    \put(130, 60){\circle*{10}}
    \put(130, 75){\makebox(0, 0)[b]{$\Gamma_{k-1}$}}
    \put(135, 60){\line(1, 0){30}}
    \put(170, 60){\circle{10}}
    \put(170, 75){\makebox(0, 0)[b]{$\Gamma_{k}$}}
    \put(65, 55){\line(0, -1){29.6}}
    \put(65, 20){\makebox(0, 0){\Large$\oslash$}}
    \put(88, 13){\makebox(0, 0)[b]{$E_{1,M}$}}
    \end{picture}
\end{center}
Therefore, the divisor $E_M^{\Delta, s}$ is effective and 
$\coeff_{\Gamma_k}E_M^{\Delta, s}=0$ if and only if $ml=sk$ holds.
In this case, we have $\max_{i}\{\coeff_{\Gamma_i}E_M^{\Delta, s}\}=l(m-s)=s(k-l)$.
\end{example}

\begin{example}\label{Sit2}
Assume that $\Delta$ satisfies the $(\nu1)$-condition such that $|\Delta|=\{P\}$.
Let $E=m_1E_1+m_2E_2$ be a non-zero effective divisor such that $E_1$ and $E_2$ 
are reduced and intersect transversally at a unique point $P=E_1\cap E_2$. 
Let $k:=\deg\Delta$ and $l_j:=\mult_P(\Delta\cap E_j)$. 
By \cite[Lemma 2.12]{N}, we may assume that $l_1=1$.  
By \cite[Lemma 2.14]{N}, we have
\[
E_M^{\Delta, s}=m_1E_{1,M}+m_2E_{2,M}+
\sum_{i=1}^{l_2}(i(m_2-s)+m_1)\Gamma_i+\sum_{i=l_2+1}^k(m_1+l_2m_2-is)\Gamma_i,
\]
where $E_{j,M}$ is the strict transform of $E_j$ in $M$. 
Moreover, the dual graph of $\phi^{-1}(E)$ is the following: 
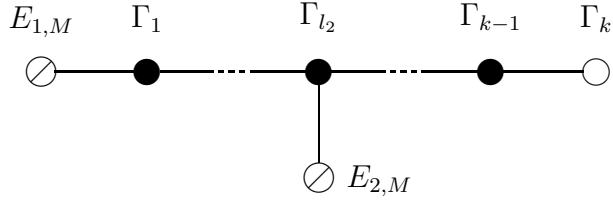
\begin{center}
    \begin{picture}(220, 90)(0, 0)
    \put(0, 60){\makebox(0, 0){\Large$\oslash$}}
    \put(0, 73){\makebox(0, 0)[b]{$E_{1,M}$}}
    \put(5, 60){\line(1, 0){30}}
    \put(40, 60){\circle*{10}}
    \put(40, 75){\makebox(0, 0)[b]{$\Gamma_{1}$}}
    \put(45, 60){\line(1, 0){20}} 
    \put(67, 60){\line(1, 0){2}}
    \put(71, 60){\line(1, 0){2}}
    \put(75, 60){\line(1, 0){2}}    
    \put(79, 60){\line(1, 0){21}}
    \put(105, 60){\circle*{10}}
    \put(105, 75){\makebox(0, 0)[b]{$\Gamma_{l_2}$}}
    \put(110, 60){\line(1, 0){20}} 
    \put(132, 60){\line(1, 0){2}}
    \put(136, 60){\line(1, 0){2}}
    \put(140, 60){\line(1, 0){2}}    
    \put(144, 60){\line(1, 0){21}}
    \put(170, 60){\circle*{10}}
    \put(170, 75){\makebox(0, 0)[b]{$\Gamma_{k-1}$}}
    \put(175, 60){\line(1, 0){30}}
    \put(210, 60){\circle{10}}
    \put(210, 75){\makebox(0, 0)[b]{$\Gamma_{k}$}}
    \put(105, 55){\line(0, -1){29.6}}
    \put(105, 20){\makebox(0, 0){\Large$\oslash$}}
    \put(128, 13){\makebox(0, 0)[b]{$E_{2,M}$}}
    \end{picture}
\end{center}
Therefore, the divisor $E_M^{\Delta, s}$ is effective and 
$\coeff_{\Gamma_k}E_M^{\Delta, s}=0$ if and only if $m_1l_1+m_2l_2=sk$ holds.
In this case, we have $\max_{i}\{\coeff_{\Gamma_i}E_M^{\Delta, s}\}=
m_1l_1+m_2l_2-sl_2=s(k-l_2)$.
\end{example}

\subsection{Curves in nonsingular surfaces}\label{singcurve_section}

In this section, we fix the notation. 

\begin{notation}\label{sing_nott}
Let $X$ be a nonsingular surface and let 
\[
M=X_k\xrightarrow{\phi_k}X_{k-1}\xrightarrow{\phi_{k-1}}X_{k-2}\to\cdots
\xrightarrow{\phi_1} X_0=X
\]
be a sequence of monoidal transforms such that the morphism 
$\phi_i\colon X_i\to X_{i-1}$ is the blowing up along $P_i\in X_{i-1}$ 
with the exceptional curve $\Gamma_i\subset X_i$. 
Let $\phi\colon M\to X$ be the composition. 
\end{notation}

\begin{definition}
We fix Notation \ref{sing_nott}.
For an effective divisor $C\subset X$, set $m_i:=\mult_{P_i}C_{i-1}$, where 
$C_{i-1}$ is the strict transform of $C$ on $X_{i-1}$. 
We call $(m_1,\dots,m_k)$ the \emph{multiplicity sequence} of $C$ with respects to 
$(\phi_1,\dots,\phi_k)$. 
If we do not need to know the decomposition $\phi=\phi_1\circ\dots\circ\phi_k$, 
then we often call $(m_1,\dots,m_k)$ the \emph{multiplicity sequence} of $C$ with respects to $\phi$.
\end{definition}

The following proposition is important in this article but is easy to prove. 
We omit the proof.

\begin{proposition}\label{sing_prop}
Fix Notation \ref{sing_nott} and assume that $X$ is projective.
\begin{enumerate}
\renewcommand{\theenumi}{\arabic{enumi}}
\renewcommand{\labelenumi}{(\theenumi)}
\item\label{sing_prop1}
Let $C\subset X$ be a reduced and irreducible curve, let $(m_1,\dots,m_k)$ 
be the multiplicity sequence of $C$ with respects to $(\phi_1,\dots,\phi_k)$. 
Let $C_M$ be the strict transform of $C$ on $M$.
\begin{enumerate}
\renewcommand{\theenumii}{\roman{enumii}}
\renewcommand{\labelenumii}{(\theenumii)}
\item\label{sing_prop11}
We have 
$
2p_a(C_M)=2p_a(C)-\sum_{i=1}^km_i(m_i-1),
$
where $p_a(\bullet)$ is the arithmetic genus.
\item\label{sing_prop12}
We have 
$
(K_{M/X}\cdot C_M)=\sum_{i=1}^km_i.
$
We note that if $C$ is nonsingular and $\phi$ is the elimination of $\Delta\subset X$ 
which satisfies the $(\nu1)$-condition, then the left side is equal to 
$\deg(\Delta\cap C)$ by \cite[Lemma 2.7]{N}.
\item\label{sing_prop13}
If $P_i\in\Gamma_{i-1}$, then $m_i\leq m_{i-1}$ holds.
\end{enumerate}
\item\label{sing_prop2}
Let $C_1$, $C_2$ be effective divisors on $X$. Let $(m_{h,1}\dots,m_{h,2})$ be the 
multiplicity sequence of $C_h$ with respects to $(\phi_1,\dots,\phi_k)$ 
for $h=1$, $2$. Let $C_{h, M}$ be the strict transform of $C_h$ on $M$ 
for $h=1$, $2$. Then we have 
$
(C_1\cdot C_2)=(C_{1, M}\cdot C_{2, M})+\sum_{i=1}^km_{1,i}m_{2,i}.
$
In particular, if $C_2$ is a nonsingular curve with $C_2\not\subset C_1$ and 
$\phi$ is the elimination of $\Delta\subset X$ which satisfies the $(\nu1)$-condition, 
then $($by \cite[Lemma 2.7]{N}$)$ there exists a sequence $1\leq i_1<\dots<i_j\leq k$ such that 
\[
m_{2,i}=
\begin{cases}
1 & \text{if }\,\,i=i_t \,\,\text{ for some }\,\,t,\\ 
0 & \text{otherwise,}
\end{cases}
\]
where $j=\deg(\Delta\cap C_2)$, thus we have 
$(C_1\cdot C_2)\geq\sum_{t=1}^jm_{1,i_t}$.
\end{enumerate}
\end{proposition}

\begin{corollary}\label{twocurve_cor}
Let $X$ be a nonsingular complete surface, $\Delta$ be a zero-dimensional 
subscheme of $X$ which satisfies the $(\nu1)$-condition, $\phi\colon M\to X$ 
be the elimination of $\Delta$ and $C_1$, $C_2$ be distinct nonsingular curves on $X$. 
We set $k:=\deg\Delta$ and $k_h:=\deg(\Delta\cap C_h)$ for $h=1$, $2$. 
Then $(C_1\cdot C_2)\geq k_1+k_2-k$ holds.
\end{corollary}

\begin{proof}
Let $(m_{h,1},\dots,m_{h,k})$ be the multiplicity sequence of $C_h$ with respect to $\phi$ 
for $h=1$, $2$. 
Then $m_{h,i}\in\{0$, $1\}$, $\sum_{i=1}^km_{h,i}=k_h$ and 
$(C_1\cdot C_2)=(C_{1,M}\cdot C_{2,M})+\sum_{i=1}^km_{1,i}m_{2,i}$ holds by 
Proposition \ref{sing_prop}, where $C_{h,M}$ is the strict transform of $C_h$ on $M$.
Thus the assertion holds since $(C_{1,M}\cdot C_{2,M})\geq 0$ and 
$\sum_{i=1}^km_{1,i}m_{2,i}\geq k_1+k_2-k$.
\end{proof}

\begin{proposition}\label{toric_prop}
Let $P\in\F_n$ and let $l_P$ be the fiber passing through $P$. 
Let $D$ be an effective divisor on $\F_n$ such that $D\in|s\sigma+tl|$. 
\begin{enumerate}
\renewcommand{\theenumi}{\arabic{enumi}}
\renewcommand{\labelenumi}{(\theenumi)}
\item\label{toric_prop1}
If $n\geq 1$ and $P\not\in\sigma$, then $\mult_PD\leq t$. 
\item\label{toric_prop2}
We have $\mult_PD\leq s+\coeff_{l_P}D$. 
\end{enumerate}
\end{proposition}

\begin{proof}
\eqref{toric_prop1}
There exists a section at infinity $\sigma_\infty$ such that $P\in\sigma_\infty$ and 
$\sigma_\infty\not\leq D$. Then $\mult_PD\leq(D\cdot\sigma_\infty)=t$ holds. 

\eqref{toric_prop2}
Let $D':=D-(\coeff_{l_P}D)D$. Then $\mult_PD'\leq(D'\cdot l_P)=s$. 
On the other hand, $\mult_PD'=\mult_PD-\coeff_{l_P}D$ holds. 
Thus $\mult_PD\leq s+\coeff_{l_P}D$ holds. 
\end{proof}

\subsection{On base point freeness for surfaces}

In this section, we prepare to prove Proposition \ref{dP-basic_prop}.

\begin{proposition}\label{tanaka_method}
Let $M$ be a nonsingular projective surface, $a$, $b$ be positive integers, 
$E$ be an effective divisor on $M$ such that $|E|$ is simple normal crossing and 
$\coeff E\subset\{1,\dots,a-1\}$, and $L$ be a nef and big divisor on $M$ such that 
$bL\sim-aK_M-E$.
Then there exists a projective and birational morphism $\alpha\colon M\to S$ 
onto a normal projective surface $S$ and an ample Cartier divisor $L_S$ on $S$ 
such that $L\sim\alpha^*L_S$ and $bL_S\sim-aK_S-\alpha_*E$ holds.
\end{proposition}

\begin{proof}
The pair $(M, (1/a)E)$ is a klt pair (see \cite[\S 2.3]{KoMo}). 
Moreover, $L-(K_M+(1/a)E)\sim_\Q(1+(b/a))L$ is a nef and big $\Q$-divisor on $M$.
Therefore there exists a positive integer $m_0$ such that for any $m\geq m_0$ 
the complete linear system $|mL|$ is base point free by the base point free theorem 
(see \cite[Theorem 3.3]{KoMo} and \cite[Theorem 0.4]{tanaka}).
Let $\alpha\colon M\to S$ be the morphism corresponds to $|mL|$ for sufficiently 
large $m$. Since $L$ is big, the morphism $\alpha$ is birational. 
Both $mL$ and $(m+1)L$ are pullbacks of ample Cartier divisors on $S$. 
Their difference $L_S$ is an ample Cartier divisor on $S$ 
such that $\alpha^*L_S\sim L$. Since $\alpha^*(bL_S)\sim-aK_M-E$, 
we have $bL_S\sim-aK_S-\alpha_*E$. 
\end{proof}

\section{Log del Pezzo surfaces}\label{dP_section}

In this section, we define the \emph{multi-index} of log del Pezzo surfaces, 
\emph{$(a,b)$-basic pairs}, and \emph{$(a,b)$-fundamental multiplets}. 
Then we see how they relate.

\subsection{Definition and properties of log del Pezzo surfaces}

\begin{definition}\label{dP_dfn}
\begin{enumerate}
\renewcommand{\theenumi}{\arabic{enumi}}
\renewcommand{\labelenumi}{(\theenumi)}
\item\label{dP_dfn1}
A normal projective surface $S$ is called a \emph{log del Pezzo surface} 
if $S$ is log-terminal and the anti-canonical divisor $-K_S$ is an ample $\Q$-Cartier 
divisor.
\item\label{dP_dfn2}
Let $S$ be a log del Pezzo surface. Set 
\[
r(S):=\sup\{r\in\Q_{>0} \, | \, -K_S\equiv rL \,\text{ for some Cartier divisor }\,
L\}.
\]
We call $r(S)$ the \emph{fractional index} of $S$. 
A Cartier divisor $L_S$ on $S$ such that $-K_S\equiv r(S)L_S$ is called a 
\emph{fundamental Cartier divisor} of $S$.
\end{enumerate}
\end{definition}

\begin{remark}\label{dP_rmk}
Any log del Pezzo surface is a rational surface by \cite[Proposition 3.6]{N}. 
In particular, the Picard group $\Pic(S)$ of $S$ is a finitely generated and torsion-free 
Abelian group. Thus $r(S)$ is $\Q$-valued, 
a fundamental Cartier divisor $L_S$ of a log del Pezzo surface $S$ is 
unique up to linearly equivalence and $-K_S\sim_\Q r(S)L_S$.
\end{remark}

\begin{definition}\label{multiind_dfn}
Let $S$ be a log del Pezzo surface, 
$a$, $b$ be positive integers and $L_S$ be the fundamental Cartier divisor of $S$.
We say that $(a, b)$ is a \emph{multi-index} of $S$ if $-aK_S\sim bL_S$ holds. 
Thus if $S$ has multi-index $(a,b)$, then $b/a=r(S)$ holds.
We say that $(a, b)$ is the \emph{normalized multi-index} of a log del Pezzo surface $S$ 
if $(a, b)$ is a multi-index of $S$ and for any multi-index $(a', b')$ of $S$ 
we have $a\leq a'$.
\end{definition}

\begin{remark}\label{multiind_rmk}
Let $S$ be a log del Pezzo surface and $L_S$ be the fundamental Cartier divisor of $S$. 
\begin{enumerate}
\renewcommand{\theenumi}{\arabic{enumi}}
\renewcommand{\labelenumi}{(\theenumi)}
\item\label{multiind_rmk1}
Let $(a, b)$ be the normalized multi-index of $S$. 
Then for any multi-index $(a', b')$ of $S$, there exists a positive integer $t$ such that 
$a'=ta$ and $b'=tb$ hold.
\item\label{multiind_rmk2}
Let $a_0$, $b_0$ be the positive integers such that 
$r(S)=b_0/a_0$ and $\gcd(a_0, b_0)=1$. 
Then $-a_0K_S\sim b_0L_S$ does \emph{not} hold in general despite 
$-a_0K_S\sim_\Q b_0L_S$. 
(See Corollary \ref{235_cor}.)
\end{enumerate}
\end{remark}

We consider the case $r(S)\geq 1$. 

\begin{proposition}\label{big_prop}
Let $S$ be a log del Pezzo surface. 
\begin{enumerate}
\renewcommand{\theenumi}{\arabic{enumi}}
\renewcommand{\labelenumi}{(\theenumi)}
\item\label{big_prop1}\cite{FJT}
If $r(S)>1$, then $S$ is isomorphic to one of the following: 
\begin{itemize}
\item
$\pr^2$ \,\, $(r(S)=3)$,
\item
$\pr^1\times\pr^1$ \,\, $(r(S)=2)$,
\item
$\pr(1,1,n)$ \,\,$(n\geq 2)$\,\, $(r(S)=(n+2)/n)$.
\end{itemize}
\item\label{big_prop2}
If $r(S)=1$, then $S$ has at most du Val singularities and such $S$ has been classified 
$($see for example \cite{brenton, demazure, HW}$)$.
\end{enumerate}
\end{proposition}

\begin{proof}
Let $S$ be a log del Pezzo surface with $r(S)\geq 1$, $L_S$ be a fundamental 
Cartier divisor of $S$. Choose positive integers $a$, $b$ such that $(a, b)$ is a 
multi-index of $S$. By assumption, $b\geq a$ holds. 
Let $\alpha\colon M\to S$ be the minimal resolution of $S$, $L_M:=\alpha^*L_S$ and 
let $E_M:=-aK_{M/S}$. 
Note that $E_M$ is a $\Z$-divisor and 
$-aK_M\sim bL_M+E_M$ holds. 
Since $\alpha$ is the minimal resolution, the divisor $E_M$ is effective 
(see \cite[Corollary 4.3]{KoMo}).
We remark that $a(K_M+L_M)\sim-E_M-(b-a)L_M$. 
Therefore, if $K_M+L_M$ is nef, then $E_M=0$ and $a=b$ hold.

Assume that $K_M+L_M$ is not nef. 
If there exists a $(-1)$-curve $\gamma$ on $M$ such that 
$(K_M+L_M\cdot \gamma)<0$, then $(L_M\cdot\gamma)=0$. However, this implies that 
$\gamma$ is $\alpha$-exceptional. This leads to a contradiction since $\alpha$ is 
the minimal resolution. 
Hence $M\simeq\pr^2$ or $\F_n$ by \cite[Theorem 2.1]{mori} and the fact $M$ is a 
nonsingular rational surface. 

\eqref{big_prop1}
Since $b>a$, $K_M+L_M$ is not nef. Thus $M\simeq\pr^2$ or $\F_n$. 
Thus the assertion follows immediately.

\eqref{big_prop2}
If $K_M+L_M$ is not nef, then $M\simeq\pr^2$ or $\F_n$. Thus $S\simeq\F_1$ 
since $r(S)=1$. However, in this case, $-K_M\sim L_M$ holds since $\alpha$ is an 
isomorphism. Thus $K_M+L_M$ is nef, a contradiction. 
Hence $K_M+L_M$ is nef. In particular, $E_M=0$. 
This implies that $S$ has canonical singularities. Therefore $S$ has at most 
du Val singularities by \cite[Theorem 4.20]{KoMo}.
\end{proof}

\subsection{$(a, b)$-basic pairs}\label{basic_section}

We are interested in log del Pezzo surfaces $S$ with $r(S)<1$, especially with 
$r(S)\in[1/2, 1)$. 
It is convenient that considering its minimal resolution and the divisor defining 
the minimal resolution in order to treat $S$. 
We introduce the following notion which is a kind of 
modification of the notion of basic pairs in the sense of Nakayama \cite[\S3]{N}.

\begin{definition}\label{basic_dfn}
Let $a$, $b$ be positive integers with $1/2\leq b/a<1$. 
A pair $(M, E)$ is called an \emph{$(a, b)$-basic pair} if the following holds: 
\begin{enumerate}
\renewcommand{\theenumi}{\arabic{enumi}}
\renewcommand{\labelenumi}{(\theenumi)}
\item\label{basic_dfn1}
$M$ is a nonsingular projective rational surface such that 
$M$ is not isomorphic to neither $\pr^2$ nor $\F_n$.
\item\label{basic_dfn2}
$E$ is a nonzero effective divisor on $M$ such that $\coeff E\subset\{1,\dots,a-1\}$ 
and $|E|$ is simple normal crossing.
\item\label{basic_dfn3}
There exists a divisor $L$ on $M$ such that $bL\sim -aK_M-E$. 
\item\label{basic_dfn4}
For any irreducible component $E_0\leq E$, we have $(L\cdot E_0)=0$.
\item\label{basic_dfn5}
$K_M+L$ is nef and $(K_M+L\cdot L)>0$ hold.
\end{enumerate}
The divisor $L$ is unique up to linearly equivalence. 
We call $L$ the \emph{fundamental divisor} of $(M, E)$.

An $(a, b)$-basic pair $(M, E)$ is called a \emph{normalized $(a, b)$-basic pair} 
if for any $t\geq 2$ with $t\mid a$ and $t\mid b$ there exists 
an irreducible component $E_0\leq E$ such that $\coeff_{E_0}E$ is not divisible by $t$.
\end{definition}

Now we see the correspondence between log del Pezzo surfaces of the normalized 
multi-index $(a, b)$ and normalized $(a, b)$-basic pairs.

\begin{proposition}\label{dP-basic_prop}
Let $a$, $b$ be positive integers with $1/2\leq b/a<1$.
\begin{enumerate}
\renewcommand{\theenumi}{\arabic{enumi}}
\renewcommand{\labelenumi}{(\theenumi)}
\item\label{dP-basic_prop1}
Let $S$ be a log del Pezzo surface of a multi-index $(a, b)$ and let $L_S$ be the 
fundamental Cartier divisor of $S$. Let $\alpha\colon M\to S$ be the minimal resolution 
of $S$ and let $E_M:=-aK_{M/S}$. 
Then the pair $(M, E_M)$ is an $(a, b)$-basic pair and the divisor $\alpha^*L_S$ is the 
fundamental divisor of $(M, E_M)$. 
If $(a, b)$ is the normalized multi-index of $S$, then $(M, E_M)$ is a normalized 
$(a, b)$-basic pair.  
\item\label{dP-basic_prop2}
Let $(M, E)$ be an $(a, b)$-basic pair and $L$ be the fundamental divisor of $(M, E)$.
Then there exists a projective and birational morphism $\alpha\colon M\to S$ such that 
$S$ is a log del Pezzo surface of a multi-index $(a, b)$ and $L\sim\alpha^*L_S$ hold, 
where $L_S$ is the fundamental Cartier divisor of $S$. 
Moreover, the morphism $\alpha$ is the minimal resolution of $S$. 
If $(M, E)$ is a normalized $(a, b)$-basic pair, then $(a, b)$ is the normalized multi-index 
of the $S$.
\item\label{dP-basic_prop3}
Let $a_i$, $b_i$ be positive integers with $1/2\leq b_i/a_i<1$
and $(M_i, E_i)$ be an $(a_i, b_i)$-basic pair 
for any $i=1$, $2$. 
Two pairs defines same log del Pezzo surface $S$ in the sense of \eqref{dP-basic_prop2} 
if and only if $M_1\simeq M_2$ and 
there exists positive integers $t_1$, $t_2$ such that 
$t_1a_1=t_2a_2$, $t_1b_1=t_2b_2$ and $t_1E_1=t_2E_2$ under the isomorphism.
\end{enumerate}
\end{proposition}

\begin{proof}
\eqref{dP-basic_prop1}
Since $\alpha$ is not an isomorphism and $S\not\simeq\pr^2$, $\pr(1,1,n)$, 
the variety $M$ is not isomorphic to neither $\pr^2$ nor $\F_n$. 
The conditions that $E_M$ is a non-zero effective divisor on $M$ such that 
$\coeff E_M\subset\{1,\dots, a-1\}$ and $|E|$ is simple normal crossing follow from 
\cite[Theorem 4.7]{KoMo}.

Let $L_M:=\alpha^*L_S$. Then $(L_M\cdot E_0)=0$ for any irreducible 
component $E_0\leq E$ since $E_0$ is $\alpha$-exceptional. 
Assume that $K_M+L_M$ is not nef. Since $M\not\simeq\pr^2$, $\F_n$, there exists 
a $(-1)$-curve $\gamma$ on $M$ such that $(K_M+L_M\cdot \gamma)<0$ by 
\cite[Theorem 2.1]{mori}. However, this implies that $\gamma$ is $\alpha$-exceptional, 
a contradiction. Thus $K_M+L_M$ is nef.

Assume that $(K_M+L_M\cdot L_M)\leq 0$. Since $L_M$ is nef and big, 
$K_M+L_M\equiv 0$ by the Hodge index theorem. 
Thus $E_M\equiv(a-b)L_M$. Hence $0=(L_M\cdot E_M)=(a-b)(L_M^2)>0$, which 
leads to a contradiction. 
Thus $(K_M+L_M\cdot L_M)>0$ holds. 
Therefore the pair $(M, E_M)$ is an $(a, b)$-basic pair and 
$L_M$ is the fundamental divisor of $(M, E_M)$. 

Assume that there exists $t\geq 2$ with $a=ta'$, $b=tb'$ ($a'$, $b'\in\Z_{>0}$) such that 
$E'_M:=(1/t)E_M$ is a $\Z$-divisor on $M$. 
Then $b'L_M\sim-a'K_M-E'_M$ holds. Thus $-a'K_S\sim b'L_S$. This means that 
$(a, b)$ is not the normalized multi-index of $S$. 

\eqref{dP-basic_prop2}
Assume that $L$ is not nef. Then there exists a curve $C$ such that 
$(L\cdot C)<0$. However, we have 
\begin{eqnarray*}
0 & > & ((a-b)L\cdot C)=((a-b)L-E\cdot C)+(E\cdot C)\\
 & = & (a(K_M+L)\cdot C)+(E\cdot C)\geq (E\cdot C). 
\end{eqnarray*}
Hence $C\leq E$, contradicts to the condition \eqref{basic_dfn4} in
Definition \ref{basic_dfn}. Thus $L$ is nef. 
On the other hand, we have
\[
0<(a(K_M+L)\cdot L)=((a-b)L-E\cdot L)=(a-b)(L^2).
\]
Hence $L$ is a nef and big divisor on $M$. Therefore 
there exists a projective and birational morphism $\alpha\colon M\to S$ 
onto a normal projective surface $S$ and an ample Cartier divisor $L_S$ on $S$ 
such that $L\sim\alpha^*L_S$ and $bL_S\sim-aK_S-\alpha_*E=-aK_S$ hold
by Proposition \ref{tanaka_method}. In particular, 
$S$ is a log del Pezzo surface, $\alpha$ is the minimal resolution of $S$ 
and $E=-aK_{M/S}$. 
Thus there exists a positive integer $s$ such that $r(S)=sb/a$. If $s\geq 2$, then 
$r(S)\geq 1$. However, in this case, $E=0$ or $M\simeq\pr^2$ or $\F_n$ 
by Proposition \ref{big_prop}, which leads to a contradiction. Therefore $r(S)=b/a$ 
and $S$ is a log del Pezzo surface of a multi-index $(a, b)$ and the fundamental 
Cartier divisor of $S$ is $L_S$. 

Assume that there exists $t\geq 2$ with $a=ta'$, $b=tb'$ ($a'$, $b'\in\Z_{>0}$) such that 
$(a', b')$ is a multi-index of $S$. Then any coefficient of $E_M$ is divisible by $t$. 
Thus $(M, E)$ is not a normalized $(a, b)$-basic pair. 

\eqref{dP-basic_prop3}
The proof is straightforward from the construction.
\end{proof}

\begin{lemma}\label{mds_lem}
Let $a$, $b$ be positive integers with $1/2\leq b/a<1$, 
$(M, E)$ be an $(a, b)$-basic pair 
and $L$ be the fundamental divisor of $(M, E)$.
\begin{enumerate}
\renewcommand{\theenumi}{\arabic{enumi}}
\renewcommand{\labelenumi}{(\theenumi)}
\item\label{mds_lem1}
Any component $C\leq E$ is a nonsingular rational curve and $(C^2)\leq -2$ hold. 
Moreover, any connected component of the dual graph of $E$ is a tree.
\item\label{mds_lem2}
If a curve $C$ on $M$ satisfies that $C\cap E\neq\emptyset$ and $(L\cdot C)=0$, 
then $C\leq E$ holds. 
\item\label{mds_lem3}
The anti-canonical divisor $-K_M$ is big and non-nef. 
In particular, $M$ is a Mori dream space $($for the definition, see \cite{testa}$)$. 
\item\label{mds_lem4}
If $a-b=1$, then the linear system $|L|$ is base point free.
\end{enumerate}
\end{lemma}

\begin{proof}
\eqref{mds_lem1}
The proof is obvious from the correspondence between Proposition \ref{dP-basic_prop} 
\eqref{dP-basic_prop1} and \eqref{dP-basic_prop2}. 

\eqref{mds_lem2}
Under the assumption, the curve $C$ is $\alpha$-exceptional and 
there exists a component $E_0\leq E$ such that $E_0$ and $C$ map $\alpha$ to 
same point, where $\alpha\colon M\to S$ is the morphism given by Proposition 
\ref{dP-basic_prop} \eqref{dP-basic_prop2}. 
Thus $C\leq E$ by \cite[Corollary 4.3]{KoMo}. 

\eqref{mds_lem3}
The anti-canonical divisor $-K_M$ is big since $-aK_M\sim bL+E$. Since $M$ is a 
nonsingular projective rational surface, $M$ is a Mori dream space by 
\cite[Theorem 1]{testa}. 
Assume that $-K_M$ is nef. Then $E_M\sim-aK_M-b\alpha^*L_S$ is $\alpha$-nef, 
where $\alpha\colon M\to S$ is the morphism and $L_S$ is the Cartier divisor on $S$ 
given in Proposition \ref{dP-basic_prop} \eqref{dP-basic_prop2}. 
Since $E$ is effective, we have $E=0$ by 
negativity lemma \cite[Lemma 3.39]{KoMo}, which leads to a contradiction. 
Thus $-K_M$ is not nef. 

\eqref{mds_lem4}
The proof is essentially same as \cite[Theorem 3.18]{N}. 

\begin{claim}[{cf.\ \cite[Lemma 3.17]{N}}]\label{bpf_claim}
The linear system $|K_M+L|$ is base point free and 
$H^1(M, m(K_M+L))=0$ hold for any $m\geq 0$.
\end{claim}

\begin{proof}[{Proof of Claim \ref{bpf_claim}}]
By running a $(-E)$-minimal model program, we can reduce to the case $M=\pr^2$ or 
$\F_n$. Hence the assertion is trivial.
\end{proof}

By Claim \ref{bpf_claim}, the complete linear system $|L-E|$ is also base point free 
since $L-E\sim a(K_M+L)$. 
Moreover, we have $H^1(M, L-E)=0$. Hence the short exact sequence 
\[
0\to \sO_M(L-E)\to \sO_M(L)\to \sO_M(L)|_E\to 0
\]
induces a short exact sequence of global sections. 
By \cite[Lemma 2.8]{N}, $|L|$ is base point free since $H^1(E, \sO_E)=0$.
\end{proof}

\begin{remark}\label{OT_rmk}
Assume that the characteristic of $\Bbbk$ is equal to zero. 
Ohashi and Taki argued the log del Pezzo surfaces $S$ in \cite{OT} 
such that each $S$ has a non-du Val 
singular point and satisfies the condition: 
\begin{itemize}
\item[$\star$]
The linear system $|-3K_S|$ contains a divisor of the form $2C$, where $C$ is a 
nonsingular curve which does not meet the singularities.
\end{itemize}
Under the assumption, since $-3K_S\sim 2C$ and $C$ is a Cartier divisor, 
the normalized multi-index of $S$ is $(3,2)$ excepts for $S\simeq\pr(1,1,6)$ 
(see Proposition \ref{big_prop}). 

Conversely, if $S$ is a log del Pezzo surface of the normalized multi-index $(3,2)$, 
then $|L_S|$ is base point free by Lemma \ref{mds_lem} \eqref{mds_lem4}, 
where $L_S$ is the fundamental Cartier divisor of $S$. 
A general member $C\in|L_S|$ is nonsingular which does not meet the singularities 
by Bertini's theorem. Therefore $2C\in|-3K_S|$ satisfies the 
condition $\star$. 
\end{remark}

\subsection{$(a, b)$-fundamental triplets}\label{triplet_section}

In order to classify $(a,b)$-basic pairs, we define the notion of 
$(a,b)$-fundamental triplets. 
The correspondence between $(a,b)$-basic pairs and $(a,b)$-fundamental triplets 
will be given in Theorem \ref{kaburi_thm}.

\begin{definition}\label{fund_dfn}
Let $a$, $b$ be positive integers with $1/2\leq b/a<1$. 
A triplet $(X, E, \Delta)$ is called an \emph{$(a, b)$-fundamental triplet} 
if the following conditions are satisfied: 
\begin{enumerate}
\renewcommand{\theenumi}{\arabic{enumi}}
\renewcommand{\labelenumi}{(\theenumi)}
\item\label{fund_dfn1}
$X$ is a nonsingular projective rational surface.
\item\label{fund_dfn2}
$\Delta$ is a nonempty zero-dimensional subscheme of $X$ 
which satisfies the $(\nu1)$-condition. 
\item\label{fund_dfn3}
$E$ is a nonzero effective divisor on $X$.
\item\label{fund_dfn4}
For any $(-1)$-curve $\gamma$ on $X$, we have $(E\cdot \gamma)\leq 0$.
\item\label{fund_dfn5}
There exists a divisor $L$ on $X$ such that 
$bL\sim-aK_X-E$ holds.
The divisor $L$ (unique up to linearly equivalence) is called the \emph{fundamental 
divisor} of $(X, E, \Delta)$. 
\item\label{fund_dfn6}
Let $\phi\colon M\to X$ be the elimination of $\Delta$ and 
let $E_M:=E_M^{\Delta, a-b}$. 
Then the pair $(M, E_M)$ is an $(a, b)$-basic pair.
\item\label{fund_dfn7}
Assume that $K_X+L$ is not big and $X\simeq\F_n$. 
Then the following conditions are satisfied: 
\begin{enumerate}
\renewcommand{\theenumii}{\roman{enumii}}
\renewcommand{\labelenumii}{(\theenumii)}
\item\label{fund_dfn71}
For a minimal section $\sigma$ of $\F_n$, 
we have $\Delta\cap\sigma=\emptyset$. 
(In particular, $n\geq 1$ holds.)
\item\label{fund_dfn72}
If there exists a component $D\leq E$ which is a section apart from $\sigma$, then 
$n+(D^2)\geq\deg(\Delta\cap D)$ holds. Moreover, if $n+(D^2)=\deg(\Delta\cap D)$, 
then $\coeff_\sigma E\geq \coeff_DE$ holds. 
\end{enumerate}
\end{enumerate}
For an $(a, b)$-fundamental triplet $(X, E, \Delta)$, the pair $(M, E_M)$ 
obtained from \eqref{fund_dfn6} is called \emph{the associated $(a, b)$-basic pair}. 
An $(a, b)$-fundamental triplet $(X, E, \Delta)$ 
is called a \emph{normalized $(a, b)$-fundamental triplet} 
if for any $t\geq 2$ with $t\mid a$ and $t\mid b$ there exists an irreducible 
component $E_0\leq E$ such that $\coeff_{E_0}E$ is not divisible by $t$.
\end{definition}

\begin{lemma}\label{normal_lem}
Let $a$, $b$ be positive integers with $1/2\leq b/a<1$ and let 
$(X, E, \Delta)$ be an $(a, b)$-fundamental triplet.
Then the triplet $(X, E, \Delta)$ is a normalized $(a, b)$-fundamental triplet 
if and only if the associated $(a, b)$-basic pair is a normalized $(a, b)$-basic pair. 
\end{lemma}

\begin{proof}
Let $(M, E_M)$ be the associated $(a, b)$-basic pair. Assume that there exists 
$t\geq 2$ with $a=ta'$, $b=tb'$ ($a'$, $b'\in\Z_{>0}$) such that $E':=(1/t)E$ is a 
$\Z$-divisor on $X$. Then $E_M=(tE')_M^{\Delta,t(a'-b')}=t({E'}_M^{\Delta,a'-b'})$. 
We note that ${E'}_M^{\Delta,a'-b'}$ is a $\Z$-divisor on $M$. 
Thus $(M, E_M)$ is not a normalized $(a, b)$-basic pair. The converse is obvious. 
\end{proof}

\begin{lemma}\label{triplet_lem}
Let $a$, $b$ be positive integers with $1/2\leq b/a<1$, let $(X, E, \Delta)$ be an 
$(a, b)$-fundamental triplet, 
let $L$ be the fundamental divisor of $(X, E, \Delta)$ and 
let $\phi\colon M\to X$ be the elimination of $\Delta$. 
Then we have the following properties:
\begin{enumerate}
\renewcommand{\theenumi}{\arabic{enumi}}
\renewcommand{\labelenumi}{(\theenumi)}
\item\label{triplet_lem1}
The divisor $L$ is nef and big, the divisor $K_X+L$ is nef and 
$(K_X+L\cdot L)>0$. Moreover, 
the divisor $L_M:=L_M^{\Delta, 1}$ is the fundamental divisor of the associated 
$(a, b)$-basic pair.
\item\label{triplet_lem2}
For any point $P\in\Delta$, we have $\mult_PE\geq a-b$. 
\item\label{triplet_lem3}
We have $(a-b)\deg\Delta=(L\cdot E)$. 
\item\label{triplet_lem4}
For any nonsingular irreducible component $E_0\leq E$, 
we have $(L\cdot E_0)=\deg(\Delta\cap E_0)$. 
\item\label{triplet_lem5}
$X$ is isomorphic to either $\pr^2$ or $\F_n$. 
Moreover, the intersection number of $E$ and $l$ is positive; 
where $l$ is a line $($if $X\simeq\pr^2$$)$, a fiber $($if $X\simeq\F_n$$)$. 
\item\label{triplet_lem6}
Let $(X, E', \Delta)$ be another $(a, b)$-fundamental triplet with $E\sim E'$. 
Then $E=E'$ holds. 
\end{enumerate}
\end{lemma}

\begin{proof}
\eqref{triplet_lem1}
Let $E_M:=E_M^{\Delta, a-b}$. 
Then we have $-aK_M-E_M\sim\phi^*(-aK_X-E)-bK_{M/X}\sim bL_M$. 
We have $K_M+L_M=\phi^*(K_X+L)$. Thus $K_X+L$ is nef and $(K_X+L\cdot L)>0$. 
Since $L_M=\phi^*L-K_{M/X}$ is nef and big (see Proposition 
\ref{dP-basic_prop}), $L$ is also nef and big. 

\eqref{triplet_lem2}
If $\mult_PE<a-b$, then $E_M$ is not effective. This is a contradiction. 
Thus $\mult_PE\geq a-b$ for any $P\in\Delta$. 

\eqref{triplet_lem3}
We have 
\begin{eqnarray*}
(L\cdot E) & =  & (L_M\cdot E_M+(a-b)K_{M/X})=(a-b)(L_M\cdot K_{M/X})\\
 & = & (a-b)(\phi^*L-K_{M/X}\cdot K_{M/X})\\
 & = & -(a-b)(K_{M/X}^2)=(a-b)\deg\Delta
\end{eqnarray*}
by \cite[Lemma 2.7]{N} and Definition \ref{basic_dfn} \eqref{basic_dfn4}. 

\eqref{triplet_lem4}
Let $E_{0,M}$ be the strict transform of $E_0$ on $M$. Then we have
\[
0=(L_M\cdot E_{0,M})=(\phi^*L-K_{M/X}\cdot E_{0,M})=(L\cdot E_0)-\deg(\Delta\cap E_0).
\]

\eqref{triplet_lem5}
The divisor $aK_X+bL\sim-E$ is not nef since $E$ is nonzero effective. 
Moreover, for any $(-1)$-curve $\gamma$ on $X$, $(aK_X+bL\cdot\gamma)\geq 0$ 
holds. Thus $X\simeq\pr^2$ or $\F_n$ and the intersection number of $aK_X+bL$
and the corresponding extremal ray is negative by \cite[Theorem 2.1]{mori}. 

\eqref{triplet_lem6}
Since the pair $(M, E_M)$ is an $(a, b)$-basic pair, there exists a 
projective and birational morphism $\alpha\colon M\to S$ such that $\alpha_*E_M=0$ 
by Proposition \ref{dP-basic_prop} \eqref{dP-basic_prop2}. 
Thus $h^0(M, E_M)=1$. Hence $E_M$ is uniquely determined by $X$, $L$ and $\Delta$. 
Therefore $E=E'$ holds since $\phi_*E_M=E$.
\end{proof}

\begin{proposition}\label{triplet_prop}
Let $a$, $b$ be positive integers with $1/2\leq b/a<1$, let $(M, E)$ be an 
$(a, b)$-basic pair and let $L$ be the fundamental divisor of $(M, E)$. 
Then there exists a projective and birational morphism $\phi\colon M\to X$ 
and a nonempty zero-dimensional subscheme $\Delta$ of $X$ such that 
the triplet $(X, \phi_*E, \Delta)$ is an $(a, b)$-fundamental triplet such that the 
associated $(a, b)$-basic pair is equal to $(M, E)$. Moreover, the fundamental divisor 
of $(X, \phi_*E, \Delta)$ is the divisor $\phi_*L$. 
\end{proposition}

\begin{proof}
\emph{Step 1.}
We recall that $M$ is a Mori dream space by Lemma \ref{mds_lem} \eqref{mds_lem3}. 
Thus we can run $(-E)$-minimal model program 
\[
\phi\colon M=M_0\to M_1\to \dots\to M_k=X
\]
and this minimal model program induces a Mori fiber space since $E$ is a nonzero 
effective divisor. In particular, $\phi_*E$ is a nonzero effective divisor.
This minimal model program is also a $(aK_M+bL)$-minimal model program. 
Thus each step $M_i\to M_{i+1}$ is the contraction of a $(-1)$-curve $\gamma_i$. 
Moreover, $(L_i\cdot \gamma_i)=1$, where $L_i$ is the push-forward of $L$ on $M_i$. 
Indeed, we have $(aK_{M_i}+bL_i\cdot\gamma_i)=-a+b(L_i\cdot\gamma_i)<0$, $a\leq 2b$ 
and $K_{M_i}+L_i$ is nef. Therefore we have $K_M+L=\phi^*(K_X+\phi_*L)$. 
We may assume that $(\phi_*E\cdot \gamma)\leq 0$ for any $(-1)$-curve $\gamma$. 
Clearly, the morphism $\phi$ is not an isomorphism since $M$ is isomorphic to 
neither $\pr^2$ nor $\F_n$. 
The anti-canonical divisor $-K_M$ is $\phi$-nef. Thus there exists a nonempty 
zero-dimensional subscheme $\Delta$ on $X$ such that $\phi$ is the elimination of 
$\Delta$ by Proposition \ref{elim_prop} \eqref{elim_prop2}. By construction, 
$L=(\phi_*L)^{\Delta, 1}$, $E=(\phi_*E)^{\Delta, a-b}$ and 
$b(\phi_*L)\sim-aK_X-\phi_*E$. 

\emph{Step 2.}
From now on, we assume that $K_X+\phi_*L$ is not big and $X=\F_n$. 
In this case, the divisor $K_X+\phi_*L$ is $\pi$-trivial, where $\pi\colon\F_n\to\pr^1$ 
is the fibration. (If $n=0$, then we may have to change the fibration structure 
$\pr^1\times\pr^1\to\pr^1$.) 
We repeat the same argument in \cite[Proposition 4.4]{N}. 

\eqref{fund_dfn71}
Assume that $\Delta\cap\sigma\neq\emptyset$. 
Let $\sigma_M$ be the strict transform of $\sigma$ on $M$. 
By \cite[Lemma 4.5]{N}, there exists a birational morphism $\phi'\colon M\to X'=\F_{n'}$ 
over $\pr^1$ with $n':=-(\sigma_M^2)=n+\deg(\Delta\cap\sigma)>n$ such that 
$\sigma_M$ is the total transform of the minimal section $\sigma'$ of $X'\to\pr^1$. 
By Proposition \ref{elim_prop} \eqref{elim_prop2}, 
there exists a nonempty zero-dimensional subscheme $\Delta'$ of $X'$ 
with $\Delta'\cap\sigma'=\emptyset$ such that 
$\phi'$ is the elimination of $\Delta'$. Thus the condition \eqref{fund_dfn71} is 
satisfied. 

\eqref{fund_dfn72}
Assume that $n>0$, $\Delta\cap\sigma=\emptyset$ and there exists a section 
$D_0\leq\phi_*E$ with $D\neq\sigma$ 
such that $(D_0^2)+n<\deg(\Delta\cap D_0)$ or $(D_0^2)+n=\deg(\Delta\cap D_0)$
and $\coeff_\sigma\phi_*E<\coeff_{D_0}\phi_*E$ holds. 
Let $n'(\geq n)$ be the maximum of $\deg(\Delta\cap D_0)-(D_0^2)$ such that 
$D_0\leq\phi_*E$ is a section. 
Let $c$ be the maximum of $\coeff_{D_0}\phi_*E$ such that 
$\deg(\Delta\cap D_0)-(D_0^2)=n'$ 
holds, where $D_0\leq\phi_*E$ is a section. 
Pick arbitrary section $D\leq\phi_*E$ 
such that $\deg(\Delta\cap D)-(D^2)=n'$
and $c=\coeff_D\phi_*E$. 
By \cite[Lemma 4.5]{N}, there exists a birational morphism $\phi'\colon M\to X'=\F_{n'}$ 
over $\pr^1$ such that 
$D_M$ is the total transform of the minimal section $\sigma'$ of $X'\to\pr^1$, 
where $D_M$ is the strict transform of $D$ on $M$. 
Then there exists a nonempty zero-dimensional subscheme $\Delta'$ of $X'$ 
with $\Delta'\cap\sigma'=\emptyset$ such that 
$\phi'$ is the elimination of $\Delta'$. 
Take an arbitrary section $D_1\leq\phi_*E$ with $D\neq D_1$. 
Then the strict transform $D_{1, M}$ on $M$ satisfies that 
$-(D_{1,M}^2)=\deg(\Delta\cap D_1)-(D_1^2)\leq n'$. 
Then $D'_1:=\phi'_*D_{1,M}$ satisfies that 
$-(D_{1,M}^2)=\deg(\Delta'\cap D'_1)-({D'}_1^2)\leq n'$. 
Moreover, if $\deg(\Delta'\cap D'_1)-({D'}_1^2)=n'$, 
then $c=\coeff_{\sigma'}\phi'_*E\geq\coeff_{D'_1}\phi'_*E$ by construction. 
Thus the condition \eqref{fund_dfn72} is satisfied. 
\end{proof}

\begin{lemma}\label{ex_lem}
Let $a$, $b$ be positive integers with $1/2\leq b/a<1$. 
Assume that a triplet $(X, E, \Delta)$ satisfies the following: 
\begin{enumerate}
\renewcommand{\theenumi}{\arabic{enumi}}
\renewcommand{\labelenumi}{(\theenumi)}
\item\label{ex_lem1}
$X$ is a nonsingular projective rational surface. 
\item\label{ex_lem2}
$\Delta$ is a zero-dimensional subscheme on $X$ which satisfies the 
$(\nu1)$-condition and $\deg\Delta\geq 2$. 
\item\label{ex_lem3}
$E$ is a nonzero effective divisor on $X$ such that $|E|$ 
is simple normal crossing and 
$\coeff E\subset\{1,\dots,a-1\}$. 
\item\label{ex_lem5}
There exists a divisor $L$ on $X$ such that $bL\sim-aK_X-E$ holds.
\item\label{ex_lem6}
The divisor $K_X+L$ is nef and $(K_X+L\cdot L)>0$ holds.
\item\label{ex_lem7}
For any irreducible component $E_1\leq E$, $(L\cdot E_1)=\deg(\Delta\cap E_1)$.
\item\label{ex_lem8}
Take any point $P\in\Delta$. 
\begin{enumerate}
\renewcommand{\theenumii}{\roman{enumii}}
\renewcommand{\labelenumii}{(\theenumii)}
\item\label{ex_lem81}
Assume that there is a unique component $E_1\leq E$ which meets $P$. 
Let $m:=\coeff_{E_1}E$, $k:=\mult_P\Delta$ and $l:=\mult_P(\Delta\cap E_1)$. 
Then $ml=(a-b)k$ and $(a-b)(k-l)<a$. 
\item\label{ex_lem82}
Assume that there is exactly two components $E_1$, $E_2\leq E$ which meet $P$. 
Let $m_j:=\coeff_{E_j}E$, $k:=\mult_P\Delta$ and $l_j:=\mult_P(\Delta\cap E_j)$. 
We may assume that $l_1=1$. 
Then $m_1l_1+m_2+l_2=(a-b)k$ and $(a-b)(k-l_2)<a$. 
\end{enumerate}
\end{enumerate}
Let $\phi\colon M\to X$ be the elimination of $\Delta$ and 
let $E_M:=E_M^{\Delta, a-b}$. Then the pair $(M, E_M)$ is an $(a,b)$-basic pair.
\end{lemma}

\begin{proof}
By the condition \eqref{ex_lem2}, $M\not\simeq\pr^2$, $\F_n$. 
Let $L_M:=L_M^{\Delta, 1}$. 
By the conditions \eqref{ex_lem8}, \eqref{ex_lem3} and \eqref{ex_lem7}, 
$E_M$ is effective, $\coeff E_M\subset\{1,\dots,a-1\}$ and 
$(L_M\cdot E_1)=0$ for any irreducible component 
$E_1\leq E_M$ (see Examples \ref{Sit1} and \ref{Sit2}). 
Since $K_M+L_M=\phi^*(K_X+L)$, $K_M+L_M$ is nef and 
$(K_M+L_M\cdot L_M)>0$. Thus $(M, E_M)$ is an $(a, b)$-basic pair. 
\end{proof}

\section{Classification of $(a, b)$-fundamental triplets}\label{classify_section}

In this section, we classify normalized $(a, b)$-fundamental triplets with 
$1/2\leq b/a<1$. 
The case $(a, b)=(2,1)$ is considered by Nakayama \cite{N}. 
Thus we are mainly interested in the case $(a, b)\neq(2,1)$. 
The main theorems are the following:

\begin{thm}\label{mainthm1}
Let $a$, $b$ be positive integers with $1/2<b/a<1$. 
The normalized $(a, b)$-fundamental triplets $(X, E, \Delta)$ are classified 
by the types defined as follows: 

\smallskip

The case $X=\pr^2:$

\begin{description}
\item[{$[$(3,2),1$]_0$}]
$(a,b)=(3,2)$, $E=l$; $l$ is a line, $\Delta\subset l$ with $\deg\Delta=4$.

\smallskip

\item[{$[$(11,7),5$]_0$}]
$(a,b)=(11,7)$, $E=5l$; $l$ is a line, $|\Delta|=\{P\}$ with $\deg\Delta=5$, 
$\deg(\Delta\cap l)=4$.

\smallskip

\item[{$[$(5,3),3$]_0$(1)}]
$(a,b)=(5,3)$, $E=3l$; $l$ is a line, $|\Delta|=\{P\}$ with $\deg\Delta=6$, 
$\deg(\Delta\cap l)=4$.

\smallskip

\item[{$[$(5,3),3$]_0$(2)}]
$(a,b)=(5,3)$, $E=3l$; $l$ is a line, $|\Delta|=\{P_1, P_2\}$ 
with $\mult_{P_i}\Delta=3$, 
$\mult_{P_i}(\Delta\cap l)=2$ $($$i=1$, $2$$)$.

\smallskip

\item[{$[$(9,5),7$]_{\times 43}$}]
$(a,b)=(9,5)$, $E=4l_1+3l_2$; $l_i$ are distinct lines, 
$\deg\Delta=7$, 
$\deg(\Delta\cap l_i)=4$ $(i=1$, $2)$ and 
$\mult_P\Delta=\mult_P(\Delta\cap l_2)=4$ 
for $P=l_1\cap l_2$.
\end{description}

The case $X=\F_n:$

\begin{description}
\item[{$[$(5,3),2;1,2$]_1$}]
$(a,b)=(5,3)$, $X=\F_2$, $E=\sigma+2l$, $\deg\Delta=3$, 
$\Delta\cap\sigma=\emptyset$ and $\Delta\subset l$.

\smallskip

\item[{$[$(7,4),2;2,4$]_1$}]
$(a,b)=(7,4)$, $X=\F_2$, $E=2\sigma+4l$, $|\Delta|=\{P\}$ with $\deg\Delta=4$, $P\not\in\sigma$ and $\deg(\Delta\cap l)=3$.

\smallskip

\item[{$[$(13,7),2;5,10$]_1$}]
$(a,b)=(13,7)$, $X=\F_2$, $E=5\sigma+10l$, $|\Delta|=\{P\}$ with $\deg\Delta=5$, $P\not\in\sigma$ and $\deg(\Delta\cap l)=3$.

\smallskip

\item[{$[$(21,11),2;9,7$]_1$}]
$(a,b)=(21,11)$, $X=\F_2$, $E=9\sigma+7l$, 
$|\Delta|=\{P\}$ with $P=\sigma\cap l$ and $\deg\Delta=\deg(\Delta\cap l)=3$.

\bigskip

\item[{$[$(2{\bi n}-1,{\bi n}+1),{\bi n};2({\bi n}-2),{\bi n}-2$]_1$}]
$n\geq 3$, $3\nmid n-2$, 
$(a,b)=(2n-1,n+1)$, $X=\F_n$, $E=(n-2)(2\sigma+l)$, $\deg\Delta=2$, 
$\Delta\cap\sigma=\emptyset$ and $\Delta\subset l$.

\smallskip

\item[{$[$(2{\bi m}+1,{\bi m}+1),3{\bi m}+2;2{\bi m},{\bi m}$]_1$}]
$m\geq 1$, 
$(a,b)=(2m+1,m+1)$, $X=\F_{3m+2}$, $E=m(2\sigma+l)$, $\deg\Delta=2$, 
$\Delta\cap\sigma=\emptyset$ and $\Delta\subset l$.

\medskip

\item[{$[$(4{\bi n}-3,2{\bi n}+1),{\bi n};4({\bi n}-2),3({\bi n}-2)$]_1$}]
$n\geq 3$, $5\nmid n-2$, 
$(a,b)=(4n-3,2n+1)$, $X=\F_n$, $E=(n-2)(4\sigma+3l)$, $|\Delta|=\{P\}$ with $\deg\Delta=3$, $P\not\in\sigma$ and $\deg(\Delta\cap l)=2$.

\smallskip

\item[{$[$(4{\bi m}+1,2{\bi m}+1),5{\bi m}+2;4{\bi m},3{\bi m}$]_1$}]
$m\geq 1$, 
$(a,b)=(4m+1,2m+1)$, $X=\F_{5m+2}$, $E=m(4\sigma+3l)$, $|\Delta|=\{P\}$ with  $\deg\Delta=3$, $P\not\in\sigma$ and $\deg(\Delta\cap l)=2$.

\medskip

\item[{$[$(2{\bi n}-2,{\bi n}),{\bi n};2({\bi n}-2),2({\bi n}-2)$]_{11}$}]
$n\geq 3$, $2\nmid n$, 
$(a,b)=(2n-2,n)$, $X=\F_n$, $E=(n-2)(2\sigma+l_1+l_2)$ $($$l_1$, $l_2:$ distinct 
fibers$)$, 
$\deg\Delta=4$, $\Delta\cap\sigma=\emptyset$, $\Delta\subset l_1\cup l_2$ 
and $\deg(\Delta\cap l_i)=2$ $(i=1,2)$.

\smallskip

\item[{$[$(2{\bi m}+1,{\bi m}+1),2{\bi m}+2;2{\bi m},2{\bi m}$]_{11}$}]
$m\geq 1$, 
$(a,b)=(2m+1,m+1)$, $X=\F_{2m+2}$, $E=m(2\sigma+l_1+l_2)$ $($$l_1$, $l_2:$ 
distinct fibers$)$, 
$\deg\Delta=4$, $\Delta\cap\sigma=\emptyset$, $\Delta\subset l_1\cup l_2$ 
and $\deg(\Delta\cap l_i)=2$ $(i=1,2)$.

\smallskip

\item[{$[$(2{\bi n}-2,{\bi n}),{\bi n};2({\bi n}-2),2({\bi n}-2)$]_1$(1)}]
$n\geq 3$, $2\nmid n$, 
$(a,b)=(2n-2,n)$, $X=\F_n$, $E=(n-2)(2\sigma+2l)$, 
$|\Delta|=\{P\}$ with $\deg\Delta=4$, $P\not\in\sigma$ and $\deg(\Delta\cap l)=2$.

\smallskip

\item[{$[$(2{\bi m}+1,{\bi m}+1),2{\bi m}+2;2{\bi m},2{\bi m}$]_1$(1)}]
$m\geq 1$, 
$(a,b)=(2m+1$, $m+1)$, $X=\F_{2m+2}$, $E=m(2\sigma+2l)$, 
$|\Delta|=\{P\}$ with $\deg\Delta=4$, $P\not\in\sigma$ and $\deg(\Delta\cap l)=2$.

\smallskip

\item[{$[$(2{\bi n}-2,{\bi n}),{\bi n};2({\bi n}-2),2({\bi n}-2)$]_1$(2)}]
$n\geq 3$, $2\nmid n$, 
$(a,b)=(2n-2,n)$, $X=\F_n$, $E=(n-2)(2\sigma+2l)$, 
$|\Delta|=\{P_1$, $P_2\}$ with $\mult_{P_i}\Delta=2$, $P_i\not\in\sigma$ and 
$\mult_{P_i}(\Delta\cap l)=1$ $(i=1$, $2)$.

\smallskip

\item[{$[$(2{\bi m}+1,{\bi m}+1),2{\bi m}+2;2{\bi m},2{\bi m}$]_1$(2)}]
$m\geq 1$, 
$(a,b)=(2m+1$, $m+1)$, $X=\F_{2m+2}$, $E=m(2\sigma+2l)$, 
$|\Delta|=\{P_1$, $P_2\}$ with $\mult_{P_i}\Delta=2$, $P_i\not\in\sigma$ and 
$\mult_{P_i}(\Delta\cap l)=1$ $(i=1$, $2)$.

\medskip

\item[{$[$(4{\bi n}-5,2{\bi n}-1),{\bi n};4({\bi n}-2),5({\bi n}-2)$]_{32}$}]
$n\geq 3$, $3\nmid n-2$, 
$(a,b)=(4n-5,2n-1)$, $X=\F_n$, $E=(n-2)(4\sigma+3l_1+2l_2)$, 
$($$l_1$, $l_2:$ distinct fibers$)$, 
$\Delta\cap\sigma=\emptyset$, $\deg\Delta=5$, $\deg(\Delta\cap l_i)=2$ $(i=1$, $2)$, 
$|\Delta|\cap l_1=\{P_1\}$ and $\mult_{P_1}\Delta=3$.

\smallskip

\item[{$[$(4{\bi m}+1,2{\bi m}+1),3{\bi m}+2;4{\bi m},5{\bi m}$]_{32}$}]
$m\geq 1$, 
$(a,b)=(4m+1$, $2m+1)$, $X=\F_{3m+2}$, $E=m(4\sigma+3l_1+2l_2)$, 
$($$l_1$, $l_2:$ distinct fibers$)$, 
$\Delta\cap\sigma=\emptyset$, $\deg\Delta=5$, $\deg(\Delta\cap l_i)=2$ $(i=1$, $2)$, 
$|\Delta|\cap l_1=\{P_1\}$ and $\mult_{P_1}\Delta=3$.

\smallskip

\item[{$[$(7,5),3;4,5$]_1$}]
$(a,b)=(7,5)$, $X=\F_3$, $E=4\sigma+5l$, 
$|\Delta|=\{P\}$ with $\deg\Delta=5$, $P\not\in\sigma$ and $\deg(\Delta\cap l)=2$.

\medskip

\item[{$[$(2{\bi n}-3,{\bi n}-1),{\bi n};2({\bi n}-2),3({\bi n}-2)$]_{111}$}]
$n\geq 3$, 
$(a,b)=(2n-3,n-1)$, $X=\F_n$, $E=(n-2)(2\sigma+l_1+l_2+l_3)$ 
$($$l_1$, $l_2$, $l_3:$ distinct fibers$)$, 
$\Delta\cap\sigma=\emptyset$, $\deg\Delta=6$ and $\deg(\Delta\cap l_i)=2$ 
$(i=1$, $2$, $3)$.

\smallskip

\item[{$[$(2{\bi n}-3,{\bi n}-1),{\bi n};2({\bi n}-2),3({\bi n}-2)$]_{21}$(1)}]
$n\geq 3$, 
$(a,b)=(2n-3,n-1)$, $X=\F_n$, $E=(n-2)(2\sigma+2l_1+l_2)$ 
$($$l_1$, $l_2:$ distinct fibers$)$, 
$\Delta\cap\sigma=\emptyset$, $\deg\Delta=6$, $\deg(\Delta\cap l_j)=2$ $(j=1$, $2)$, 
$|\Delta|\cap l_1=\{P\}$ and $\mult_{P}\Delta=4$. 

\smallskip

\item[{$[$(2{\bi n}-3,{\bi n}-1),{\bi n};2({\bi n}-2),3({\bi n}-2)$]_{21}$(2)}]
$n\geq 3$, 
$(a,b)=(2n-3,n-1)$, $X=\F_n$, $E=(n-2)(2\sigma+2l_1+l_2)$ 
$($$l_1$, $l_2:$ distinct fibers$)$, 
$\Delta\cap\sigma=\emptyset$, $\deg\Delta=6$, $\deg(\Delta\cap l_j)=2$ $(j=1$, $2)$, 
$|\Delta|\cap l_1=\{P_1$, $P_2\}$ and $\mult_{P_i}\Delta=2$ $(i=1$, $2)$. 

\smallskip

\item[{$[$(4{\bi n}-6,2{\bi n}-2),{\bi n};4({\bi n}-2),6({\bi n}-2)$]_{11}$}]
$n\geq 3$, $2\nmid n$, 
$(a,b)=(4n-6,2n-2)$, $X=\F_n$, $E=(n-2)(4\sigma+3l_1+3l_2)$ 
$($$l_1$, $l_2:$ distinct fibers$)$, 
$\Delta\cap\sigma=\emptyset$, $\deg\Delta=6$, 
$|\Delta|\cap l_i=\{P_i\}$ with $\mult_{P_i}\Delta=3$ and 
$\mult_{P_i}(\Delta\cap l_i)=2$ $(i=1$, $2)$. 

\smallskip

\item[{$[$(4{\bi m}+1,2{\bi m}+1),2{\bi m}+2;4{\bi m},6{\bi m}$]_{11}$}]
$m\geq 1$, 
$(a,b)=(4m+1$, $2m+1)$, $X=\F_{2m+2}$, $E=m(4\sigma+3l_1+3l_2)$ 
$($$l_1$, $l_2:$ distinct fibers$)$, 
$\Delta\cap\sigma=\emptyset$, $\deg\Delta=6$, 
$|\Delta|\cap l_i=\{P_i\}$ with $\mult_{P_i}\Delta=3$ and 
$\mult_{P_i}(\Delta\cap l_i)=2$ $(i=1$, $2)$. 

\smallskip

\item[{$[$(3,2),3;2,3$]_{1\infty}$}]
$(a,b)=(3,2)$, $X=\F_3$, $E=\sigma+\sigma_\infty$, 
$\deg\Delta=6$ and $\Delta\subset\sigma_\infty$.

\medskip

\item[{$[$(4{\bi n}-7,2{\bi n}-3),{\bi n};4({\bi n}-2),7({\bi n}-2)$]_{322}$}]
$n\geq 3$, 
$(a,b)=(4n-7,2n-3)$, $X=\F_n$, $E=(n-2)(4\sigma+3l_1+2l_2+2l_3)$, 
$($$l_1$, $l_2$, $l_3:$ distinct fibers$)$, 
$\Delta\cap\sigma=\emptyset$, $\deg\Delta=7$, $\deg(\Delta\cap l_i)=2$ 
$(i=1, 2, 3)$ and $|\Delta|\cap l_1=\{P_1\}$ with $\mult_{P_1}\Delta=3$. 

\smallskip

\item[{$[$(4{\bi n}-7,2{\bi n}-3),{\bi n};4({\bi n}-2),7({\bi n}-2)$]_{43}$(1)}]
$n\geq 3$, 
$(a,b)=(4n-7$, $2n-3)$, $X=\F_n$, $E=(n-2)(4\sigma+4l_1+3l_2)$, 
$($$l_1$, $l_2:$ distinct fibers$)$, 
$\Delta\cap\sigma=\emptyset$, $\deg\Delta=7$, $\deg(\Delta\cap l_i)=2$, 
$|\Delta|\cap l_i=\{P_i\}$ $(i=1$, $2)$ 
with $\mult_{P_1}\Delta=4$ and $\mult_{P_2}\Delta=3$. 

\smallskip

\item[{$[$(4{\bi n}-7,2{\bi n}-3),{\bi n};4({\bi n}-2),7({\bi n}-2)$]_{43}$(2)}]
$n\geq 3$, 
$(a,b)=(4n-7$, $2n-3)$, $X=\F_n$, $E=(n-2)(4\sigma+4l_1+3l_2)$, 
$($$l_1$, $l_2:$ distinct fibers$)$, 
$\Delta\cap\sigma=\emptyset$, $\deg\Delta=7$, $\deg(\Delta\cap l_i)=2$ $(i=1$, $2)$, 
$|\Delta|\cap l_1=\{P_{11}$, $P_{12}\}$ with $\mult_{P_{1j}}\Delta=2$ $(j=1$, $2)$ and 
$|\Delta|\cap l_2=\{P_2\}$ with $\mult_{P_2}\Delta=3$. 

\smallskip

\item[{$[$(15,9),3;12,21$]_{5\infty 1}$}]
$(a,b)=(15,9)$, $X=\F_3$, $E=7\sigma+5\sigma_\infty+6l$, 
$\deg\Delta=7$, $\Delta\cap\sigma=\emptyset$, $\deg(\Delta\cap l)=2$ and 
$\mult_P\Delta=\mult_P(\Delta\cap\sigma_\infty)=6$ 
for $P=\sigma_\infty\cap l$.

\smallskip

\item[{$[$(5,3),3;4,7$]_{2\infty 1}$}]
$(a,b)=(5,3)$, $X=\F_3$, $E=2\sigma+2\sigma_\infty+l$, 
$\deg\Delta=7$, $\deg(\Delta\cap l)=2$, 
$\deg(\Delta\cap\sigma_\infty)=6$ and 
$\mult_P\Delta=\mult_P(\Delta\cap l)=2$ 
for $P=\sigma_\infty\cap l$.
\end{description}

The symbol {\rm$[$\textbf{({\bi a,b}),{\bi e}}$]$} indicates that 
the corresponding triplet is a normalized $(a, b)$-fundamental triplet and 
$X\simeq\pr^2$, $E\sim el$ $(l$ is a line$)$. 
The subscripts ${}_0$, ${}_{\times st}$ have the following meaning: 

\begin{itemize}
\item[{${}_0:$}]
$|E|$ is a line.
\item[{${}_{\times st}:$}]
$|E|$ is the union of two lines $l_1$ and $l_2$ and the ratio $\coeff_{l_1}E:\coeff_{l_2}E$ 
is equal to $s:t$.
\end{itemize}

The symbol {\rm$[$\textbf{({\bi a,b}),{\bi n};{\bi d,e}}$]$} indicates that 
the corresponding triplet is a normalized $(a, b)$-fundamental triplet and 
$X\simeq\F_n$, $E\sim d\sigma+el$.
The subscripts ${}_1$, ${}_{st}$, ${}_{stu}$, ${}_{r\infty}$, ${}_{r\infty 1}$, ${}_{r\infty st}$ 
have the following meaning:

\begin{itemize}
\item[{${}_1:$}]
$|E|$ is the union of $\sigma$ and $l$.
\item[{${}_{st}:$}]
$|E|$ is the union of $\sigma$ and two fibers $l_1$ and $l_2$. The ratio 
$\coeff_{l_1}E:\coeff_{l_2}E$ is equal to $s:t$.
\item[{${}_{stu}:$}]
$|E|$ is the union of $\sigma$ and three fibers $l_1$, $l_2$ and $l_3$. The ratio 
$\coeff_{l_1}E:\coeff_{l_2}E:\coeff_{l_3}E$ is equal to $s:t:u$.
\item[{${}_{r\infty}:$}]
$|E|$ is the union of $\sigma$ and $\sigma_\infty$ such that $\coeff_{\sigma_\infty}E=r$.
\item[{${}_{r\infty 1}:$}]
$|E|$ is the union of $\sigma$, $\sigma_\infty$ and $l$ such that  
$\coeff_{\sigma_\infty}E=r$.
\item[{${}_{r\infty st}:$}]
$|E|$ is the union of $\sigma$, $\sigma_\infty$ and two fibers $l_1$ and $l_2$ 
such that $\coeff_{\sigma_\infty}E=r$. The ratio 
$\coeff_{l_1}E:\coeff_{l_2}E$ is equal to $s:t$.
\end{itemize}
\end{thm}

\begin{thm}\label{mainthm2}
Let $a$, $b$ be positive integers with $1/2=b/a$ and $b>1$. 
The normalized $(a, b)$-fundamental triplets $(X, E, \Delta)$ are classified 
by the types defined as follows: 

\smallskip

The case $X=\pr^2:$

\begin{description}
\item[{$[$(6,3),6$]_{\times 21}$}]
$(a,b)=(6,3)$, $E=4l_1+2l_2$; $l_i$ are distinct lines, 
$\deg\Delta=8$, $|\Delta|=\{P$, $P_1\}$, 
where $P=l_1\cap l_2$, $\mult_P\Delta=\mult_P(\Delta\cap l_2)=4$, 
$P_1\in l_1\setminus\{P\}$, $\mult_{P_1}\Delta=4$ and $\mult_{P_1}(\Delta\cap l_1)=3$. 
\end{description}

The case $X=\F_3:$

\begin{description}
\item[{$[$(6,3),3;6,12$]_{2\infty 11}$}]
$(a,b)=(6,3)$, $E=4\sigma+2\sigma_\infty+3l_1+3l_2$ 
$($$l_1$, $l_2:$ distinct fibers$)$, 
$\deg\Delta=8$, $\Delta\cap\sigma=\emptyset$, 
$\deg(\Delta\cap\sigma_\infty)=6$, $\deg(\Delta\cap l_i)=2$ and $\mult_{P_i}\Delta=
\mult_{P_i}(\Delta\cap\sigma_\infty)=3$ for $P_i=\sigma_\infty\cap l_i$ $(i=1$, $2)$. 

\smallskip

\item[{$[$(10,5),3;10,20$]_{4\infty 53}$}]
$(a,b)=(10,5)$, $E=6\sigma+4\sigma_\infty+5l_1+3l_2$ 
$($$l_1$, $l_2:$ distinct fibers$)$, 
$\deg\Delta=8$, $\Delta\cap\sigma=\emptyset$, 
$\deg(\Delta\cap\sigma_\infty)=6$, $\deg(\Delta\cap l_i)=2$ 
and $\mult_{P_1}\Delta=\mult_{P_1}(\Delta\cap\sigma_\infty)=5$, 
$\mult_{P_2}\Delta=\mult_{P_2}(\Delta\cap l_2)=2$ for 
$P_i=\sigma_\infty\cap l_i$ $(i=1$, $2)$. 

\smallskip

\item[{$[$(4,2),3;4,8$]_{2\infty 11}$}]
$(a,b)=(4,2)$, $E=2\sigma+2\sigma_\infty+l_1+l_2$ 
$($$l_1$, $l_2:$ distinct fibers$)$, 
$\deg\Delta=8$, 
$\deg(\Delta\cap\sigma_\infty)=6$, 
$|\Delta|\cap l_i=\{P_i\}$ with $P_i=\sigma_\infty\cap l_i$ and 
$\mult_{P_i}\Delta=\mult_{P_i}(\Delta\cap l_i)=2$ $(i=1$, $2)$.
\end{description}

The meanings of the symbols of the types are same as the meaning given in 
Theorem \ref{mainthm1}.
\end{thm}

As immediate corollaries, we have the following results.

\begin{corollary}\label{maincor}
We have
\[
\FS_2\cap\left(1/2, 1\right)=\left\{\frac{2s+t}{4s+t}\,\,
\bigg|\,\, s\in\Z_{>0},\,\, t\in\{4,5,6\}\right\}, 
\]
where $\FS_2$ is the set of all the fractional indices of log del Pezzo surfaces. 
\end{corollary}

\begin{corollary}\label{235_cor}
\begin{enumerate}
\renewcommand{\theenumi}{\arabic{enumi}}
\renewcommand{\labelenumi}{(\theenumi)}
\item\label{235_cor1}
Consider the log del Pezzo surface $S$ corresponding to a fundamental triplet of type 

\rm{\textbf{$[$(4{\bi n}-6,2{\bi n}-2),{\bi n};4({\bi n}-2),6({\bi n}-2)$]_{11}$}}, 

\rm{\textbf{$[$(15,9),3;12,21$]_{5\infty 1}$}}, 
\rm{\textbf{$[$(6,3),6$]_{\times 21}$}}, 
\rm{\textbf{$[$(6,3),3;6,12$]_{2\infty 11}$}}, 

\rm{\textbf{$[$(10,5),3;10,20$]_{4\infty 53}$}} or 
\rm{\textbf{{$[$(4,2),3;4,8$]_{2\infty 11}$}}}. 

Let $a_0$, $b_0$ be the positive integers such that $r(S)=b_0/a_0$ and 
$\gcd(a_0, b_0)=1$. 
Then $-a_0K_S\not\sim b_0L_S$. 
\item\label{235_cor2}
Let $a$, $b$ be positive integers with $1/2\leq b/a<1$. Take an arbitrary log del Pezzo 
surface $S$ with $r(S)=b/a$. Then one of $-2aK_S$, $-3aK_S$ or $-5aK_S$ 
is a Cartier divisor. 
\end{enumerate}
\end{corollary}

We start to prove Theorems \ref{mainthm1} and \ref{mainthm2}. 
We remark that any of the triplet in Theorems \ref{mainthm1} and \ref{mainthm2} 
is a normalized $(a, b)$-fundamental triplet by Lemma \ref{ex_lem}.
Let $a$, $b$ be positive integers with $1/2\leq b/a<1$. 
Let $(X, E, \Delta)$ be an $(a, b)$-fundamental triplet,
$L$ be the fundamental divisor of $(X, E, \Delta)$, 
$\phi\colon M\to X$ be the elimination of $\Delta$, 
$E_M:=E_M^{\Delta,a-b}$, $L_M:=L_M^{\Delta,1}$ 
and let $k:=\deg\Delta$. 
We may assume that if $a=2b$ then $\coeff(E/b)\not\subset\Z$ by 
Lemma \ref{normal_lem} and \cite{N}. 
By Lemma \ref{ex_lem} \eqref{ex_lem1}, $X$ is isomorphic to 
either $\pr^2$ or $\F_n$.

\subsection{The case $X=\mathbb{P}^2$}\label{p2_section}

We consider the case $X=\pr^2$. 
Let $e$ and $h$ be the positive integers determined by $E\sim el$ and $L\sim hl$, 
where $l$ is a line on $X$. Then we have 
$e=3a-hb$, $h\geq 4$ and $eh=(a-b)k$. 
Since $e\geq 1$ and $1/2\leq b/a$, we have $h<6$. Thus $h=4$ or $5$.

\begin{claim}\label{mult_claim}
\begin{enumerate}
\renewcommand{\theenumi}{\arabic{enumi}}
\renewcommand{\labelenumi}{(\theenumi)}
\item\label{mult_claim01} 
The triplet $(h,k,b/a)$ is one of 
$(5,5,1/2),(4,4,2/3)$, $(4,5,7/11)$, $(4,6,3/5)$, $(4,7,5/9)$ or $(4,8,1/2)$.
\item\label{mult_claim02}
Take any irreducible curve $C$ on $X$ such that $(L_M\cdot\phi^{-1}_*C)=0$. 
$($We note that any component of $E$ satisfies this condition.$)$
If $(h,k,b/a)=(4,8,1/2)$, then $C$ is either a conic or a line. 
If $(h,k,b/a)\neq(4,8,1/2)$, then $C$ is a line. 
\end{enumerate}
\end{claim}

\begin{proof}
Let $m$ be the degree of $C$ on $\pr^2$.
$C$ is a rational curve by Proposition \ref{dP-basic_prop} and \cite[\S 4]{KoMo}. 
Let $(m_1,\dots,m_k)$ be the multiplicity sequence of $C$ with respects to $\phi$. 
By Proposition \ref{sing_prop}, we have the following:
\begin{enumerate}
\renewcommand{\theenumi}{\roman{enumi}}
\renewcommand{\labelenumi}{(\theenumi)}
\item\label{mult_claim1}
$0\leq m_i\leq m$ for any $i$.
\item\label{mult_claim2}
$hm=\sum_{i=1}^km_i$.
\item\label{mult_claim3}
$m^2-3m+2=\sum_{i=1}^km_i(m_i-1)$.
\end{enumerate}

By \eqref{mult_claim1} and \eqref{mult_claim2}, we have $h\leq k$. 
Since $(3a-hb)h=(a-b)k$, we have $2a\geq(h-1)b$. 
Thus if $h=5$, then $b/a=1/2$; if $h=4$, then $b/a\leq 2/3$. 

Assume that $h=5$. 
Then $b/a=1/2$ and $k=5$ holds. By \eqref{mult_claim2}, $m=m_i$ for any $i$. 
Thus $(m-1)(m-2)=5m(m-1)$ by \eqref{mult_claim3}. Hence $m=1$ holds. 

Assume that $h=4$. We know that $b/a\in[1/2,2/3]$. Since $k=4(3a-4b)/(a-b)$, 
we have $4\leq k\leq 8$. 
On the other hand, we have
\[
m^2+m+2=\sum_{i=1}^km_i^2\geq k\Bigl(\sum_{i=1}^km_i/k\Bigr)^2=16m^2/k.
\]
Thus we have the inequality $f(m)\leq 0$, 
where $f(x):=(16-k)x^2-kx-2k$.
If $m=1$, then $4\leq k\leq 8$ since $f(1)=16-4k$. 
If $m=2$, then $k=8$ since $f(2)=64-8k$. 
The value $m$ cannot be bigger than $2$ since $f'(x)>0$ for any $x\geq 1$ (note that 
$k/(2(16-k))\leq 1/2$). 
Therefore the assertion holds since $b/a=(12-k)/(16-k)$ if $h=4$. 
\end{proof}

\begin{claim}\label{twoline_claim}
Assume that there exist two distinct lines $l_1$, $l_2\leq E$. 
We also assume that $E$ contains no nonsingular conic. 
Then $|E|$ consists of exactly two components $l_1$, $l_2$. Moreover, 
the triplet $(h,k,b/a)$ is either $(4,7,5/9)$ or $(4,8,1/2)$.
\end{claim}

\begin{proof}
Let $P:=l_1\cap l_2$. We can assume that $\mult_P(\Delta\cap l_1)\leq 1$ by 
Example \ref{Sit2}. Since $\deg(\Delta\cap l_i)=(L\cdot l_i)=h$ for $i=1$, $2$, 
we have $k\geq 2h-1$. Thus $(h,k,b/a)$ is either $(4,7,5/9)$ or $(4,8,1/2)$.
Using the same argument, if there exists another line $l_3\leq E$, 
then $k\geq 3h-3$ holds. This leads to a contradiction. 
Thus $|E|=l_1\cup l_2$. 
\end{proof}

\textbf{The case $(h,k,b/a)=(5,5,1/2)$:}
In this case, we have 
$E=bl$ for some line $l$ by Claims \ref{mult_claim} and \ref{twoline_claim}.
This contradicts to the assumption $\coeff(E/b)\not\subset\Z$. 

\textbf{The case $(h,k,b/a)=(4,4,2/3)$:}
In this case, there exists a positive integer $t$ such that $a=3t$, $b=2t$ and $E=tl$ 
for some line $l$ by Claims \ref{mult_claim} and \ref{twoline_claim}.
Since $\deg(\Delta\cap l)=4$, we have $\Delta\subset l$. 
If $t=1$, then this triplet is nothing but the type \textbf{$[$(3,2),1$]_0$}.

\textbf{The case $(h,k,b/a)=(4,5,7/11)$:}
In this case, there exists a positive integer $t$ such that $a=11t$, $b=7t$ and $E=5tl$ 
for some line $l$ by Claims \ref{mult_claim} and \ref{twoline_claim}.
We note that $\deg(\Delta\cap l)=4$. 
For any $P\in\Delta$, we have $5t\mult_P(\Delta\cap l)=4t\mult_P\Delta$ by 
Example \ref{Sit1}. Thus $\mult_P(\Delta\cap l)=4$ and $\mult_P\Delta=5$ hold. 
If $t=1$, then this triplet is nothing but the type \textbf{$[$(11,7),5$]_0$}.

\textbf{The case $(h,k,b/a)=(4,6,3/5)$:}
In this case, there exists a positive integer $t$ such that $a=5t$, $b=3t$ and $E=3tl$ 
for some line $l$ by Claims \ref{mult_claim} and \ref{twoline_claim}.
We note that $\deg(\Delta\cap l)=4$. 
For any $P\in\Delta$, we have $3t\mult_P(\Delta\cap l)=2t\mult_P\Delta$ 
by Example \ref{Sit1}. 
Thus the pair $(\mult_P\Delta, \mult_P(\Delta\cap l))$ is either $(6,4)$ or $(3,2)$. 
If $|\Delta|=\{P\}$, then 
$(\deg\Delta, \deg(\Delta\cap l))=(6,4)$. 
If $t=1$, then this triplet is nothing but the type \textbf{$[$(5,3),3$]_0$(1)}.
If $|\Delta|=\{P_1$, $P_2\}$, then $(\mult_{P_i}\Delta, \mult_{P_i}(\Delta\cap l))=(3,2)$. 
If $t=1$, then this triplet is nothing but the type \textbf{$[$(5,3),3$]_0$(2)}.

\textbf{The case $(h,k,b/a)=(4,7,5/9)$:}
In this case, there exists a positive integer $t$ such that $a=9t$, $b=5t$ and 
$E\sim 7tl$. 
Assume that $E=7tl$ for some line $l$. Then for any $P\in\Delta$ we have 
$7t\mult_P(\Delta\cap l)=4t\mult_P\Delta$ by Example \ref{Sit1}.
Thus $\mult_P(\Delta\cap l)=4$ and $\mult_P\Delta=7$ hold. 
However, in this case $\coeff E_M\not\subset\{1,\dots,a-1\}$ by Example \ref{Sit1}, 
a contradiction. 
Thus there exist distinct lines $l_1$, $l_2$ and positive integers $c_1$, $c_2$ such that 
$E=c_1l_1+c_2l_2$ (hence $c_1+c_2=7t$) by Claims \ref{mult_claim} and \ref{twoline_claim}.
Let $P:=l_1\cap l_2$. If $P\not\in\Delta$, then 
$7=k\geq\sum_{i=1,2}\deg(\Delta\cap l_i)=8$, a contradiction. 
Hence $P\in\Delta$. We can assume that $\mult_P(\Delta\cap l_1)=1$. 
We have 
$4t\mult_P\Delta=c_1+c_2\mult_P(\Delta\cap l_2)$ by Example \ref{Sit2}. 
Since $3=\deg(\Delta\cap l_1\setminus\{P\})\leq\deg(\Delta\setminus l_2)\leq 7-4=3$, 
$\mult_{P_1}\Delta=\mult_{P_1}(\Delta\cap l_1)$ holds for any $P_1\in\Delta\cap l_1$ 
such that $P_1\neq P$. Thus $(c_1, c_2)=(4t, 3t)$ by Example \ref{Sit1}. 
In particular, $\mult_P\Delta=\mult_P(\Delta\cap l_2)=4$. 
If $t=1$, then this triplet is nothing but 
the type \textbf{$[$(9,5),7$]_{\times 43}$}.

\textbf{The case $(h,k,b/a)=(4,8,1/2)$:}
In this case, there exists a positive integer $t$ such that $a=2t$, $b=t$ and 
$E\sim 2tl$. 
Assume that there exists a nonsingular conic $C\leq E$. 
Then $\Delta\subset C$ since $\deg(\Delta\cap C)=8$. Conversely, for any 
nonsingular conic $C$ and any subscheme $\Delta\subset C$ 
with $\deg\Delta=8$, then the 
triplet $(X, tC, \Delta)$ is a $(2t, t)$-fundamental triplet. Thus $E$ must be equal to 
$tC$ by Lemma \ref{triplet_lem} \eqref{triplet_lem6}. This contradicts to the assumption 
$\coeff(E/b)\not\subset\Z$. 
Thus $E=c_1l_1+c_2l_2$ ($c_1$, $c_2$ positive integers with $t\nmid c_1, c_2$ 
and $c_1+c_2=2t$) 
such that $l_1$, $l_2$ are distinct lines 
by Claims \ref{mult_claim} and \ref{twoline_claim}. 
We may assume that $2t>c_1>t>c_2>0$. Let $P:=l_1\cap l_2$.
Assume that $P\not\in\Delta$. 
Then $8=\deg\Delta\geq\sum_{i=1,2}\deg(\Delta\cap l_i)=8$. 
Thus $\mult_{P_1}(\Delta\cap l_1)=\mult_{P_1}\Delta$ holds for any $P_1\in\Delta\setminus\{P\}$. Hence $c_1=t$ holds by Example \ref{Sit1}, 
a contradiction. 
Therefore $P\in\Delta$. Since $c_2<t=a-b$, we have $\mult_P(\Delta\cap l_1)=1$. 
(Indeed, $|\Delta|\cap l_2=\{P\}$ holds.) 
Moreover, $\mult_P(\Delta\cap l_2)=4$ holds. 
Thus $t\mult_P\Delta=c_1+4c_2$ holds by Example \ref{Sit2}. 
On the other hand, for any point $P_1\in\Delta\setminus\{P\}$, we have 
$c_1\mult_{P_1}(\Delta\cap l_1)=t\mult_{P_1}\Delta$ by Example \ref{Sit1}. 
Since $\mult_{P_1}\Delta\leq k-\deg(\Delta\cap l_2)=4$, we have 
$(\mult_{P_1}\Delta, \mult_{P_1}(\Delta\cap l_1))=(4, 3)$ and $(c_1, c_2)=(4t/3, 2t/3)$. 
If $t=3$, then this triplet is nothing but the type 
\textbf{$[$(6,3),6$]_{\times 21}$}.

\subsection{The case $X=\mathbb{F}_n$ with $K_X+L$ big}\label{big_section}

We consider the case $X=\F_n$ such that the divisor $K_X+L$ is big. 

Let $e_0$, $e$, $h_0$, $h$ be the nonnegative integers determined by 
$E\sim e_0\sigma+el$ and $L\sim h_0\sigma+hl$. 
We have $e_0=2a-h_0b$ and $e=(n+2)a-hb$. Since $e_0\geq 1$, 
we have $h_0\leq 3$ and $b/a<2/3$. 
Moreover, the divisor $K_X+L\sim(h_0-2)\sigma+(h-(n+2))l$ is nef 
and big. Thus $h_0=3$ and $\max\{3n, 2n+2\}\leq h\leq(n+2)a/b$ holds. 
Since $k(a-b)=(L\cdot E)=(-3n+6+2h)a+(9n-6h)b$, we have the following:
\[
\frac{b}{a}=\frac{2h-3n+6-k}{6h-9n-k}.
\]

\begin{claim}\label{Fn_coeff_claim}
\begin{enumerate}
\renewcommand{\theenumi}{\arabic{enumi}}
\renewcommand{\labelenumi}{(\theenumi)}
\item\label{Fn_coeff_claim1}
For any point $P\in\Delta$, we have $\coeff_{l_P}E\geq 2b-a$, where 
$l_P$ is the fiber passing through $P$. 
$($In particular, if $b/a>1/2$, then $\coeff_{l_P}E>0$ holds.$)$
\item\label{Fn_coeff_claim2}
For any fiber $l\leq E$, we have $\coeff_lE\geq a/3$.
\item\label{Fn_coeff_claim3}
If $b/a>1/2$, then $b/a\leq(3n+5)/(3h)$ holds.
\end{enumerate}
\end{claim}

\begin{proof}
\eqref{Fn_coeff_claim1}
By Proposition \ref{toric_prop}, 
$\coeff_{l_P}E\geq(2a-3b)-(a-b)=a-2b$. 

\eqref{Fn_coeff_claim2}
Let $l_M$ be the strict transform of $l$ on $M$. 
Since $\deg(\Delta\cap l)=(L\cdot l)=3$, we have $(l_M^2)=-3$. 
Thus 
\begin{eqnarray*}
-a & = & (-aK_M\cdot l_M)=(-aK_M-bL_M\cdot l_M)\\
 & = & (E_M\cdot l_M)\geq(-3)\coeff_lE.
\end{eqnarray*}

\eqref{Fn_coeff_claim3}
Follows from 
\eqref{Fn_coeff_claim1}, \eqref{Fn_coeff_claim2} and the fact $\coeff_lE\leq e
=(n+2)a-hb$ for any fiber $l\leq E$.
\end{proof}

\begin{claim}\label{big_possible_claim}
\begin{enumerate}
\renewcommand{\theenumi}{\arabic{enumi}}
\renewcommand{\labelenumi}{(\theenumi)}
\item\label{big_possible_claim1}
If $b/a>1/2$, then the tetrad $(n, h, b/a, k)$ is one of the following: 
\[
(2,6,3/5,3), (2,6,4/7,4), (2,6,7/13,5), (2,7,11/21,3).
\]
\item\label{big_possible_claim2}
If $b/a=1/2$, then the triplet $(n, h, k)$ is one of the following: 
\begin{eqnarray*}
& & (0,3,6), (0,4,4), (1,5,5), (1,6,3), (2,6,6),\\
&  & (2,7,4), (2,8,2), (3,9,3), (3,10,1).
\end{eqnarray*}
\end{enumerate}
\end{claim}

\begin{proof}
We can show that $n\leq 3$. Indeed, if $n\geq 4$, then $n=4$, $h=12$ and $b/a=1/2$ 
since $3n\leq h\leq(n+2)a/b\leq 2(n+2)$. Thus $k=0$ holds since $b/a=(18-k)/(36-k)$. 
This leads to a contradiction. 

\textbf{The case $n=0$:}
Since $K_X+L$ is big, we have $h\geq 3$. Moreover, $h\leq 4$ since $h\leq 2a/b\leq 4$. 
If $h=4$, then $b/a=(14-k)/(24-k)=1/2$. Thus $(n, h, b/a, k)=(0, 4, 1/2, 4)$. 
If $h=3$, then $b/a=(12-k)/(18-k)\leq 5/9$ by Claim \ref{Fn_coeff_claim} 
\eqref{Fn_coeff_claim3}. 
Thus $(n, h, b/a, k)=(0, 3, 7/13, 5)$ or $(0, 3, 1/2, 6)$.
Assume that $(n, h, b/a, k)=(0, 3, 7/13, 5)$. 
Then for any point $P\in\Delta$, $\coeff_{\sigma_P}E, \coeff_{l_P}E>0$
by Claim \ref{Fn_coeff_claim} \eqref{Fn_coeff_claim1}, where $\sigma_P$ (resp.\ $l_P$) 
is the minimal section (resp.\ fiber) passing through $P$. 
We can assume that $\mult_P(\Delta\cap l_P)=1$. Since $\deg(\Delta\cap l_P)=3$, 
there exists a point $Q\in\Delta\cap l_P\setminus\{P\}$. Then
$\coeff_{\sigma_Q}E>0$. 
However, $5=k\geq\deg(\Delta\cap\sigma_P)+\deg(\Delta\cap\sigma_Q)=6$, 
which leads to a contradiction. Thus $(n, h, b/a, k)\neq(0, 3, 7/13, 5)$. 

\textbf{The case $n=1$:}
Since $(E\cdot\sigma)\leq 0$, we have $a\leq(h-3)b$. Hence 
$5\leq h\leq 3a/b\leq 6$. 
If $h=6$, then $b/a=(15-k)/(27-k)=1/2$. Thus $(n, h, b/a, k)=(1, 6, 1/2, 3)$. 
If $h=5$, then $b/a=(13-k)/(21-k)\leq 8/15$ by Claim \ref{Fn_coeff_claim} 
\eqref{Fn_coeff_claim3}. 
Thus $(n, h, b/a, k)=(1, 5, 9/17, 4)$ or $(1, 5, 1/2, 5)$. 
Assume that $(n, h, b/a, k)=(1, 5, 9/17, 4)$. 
Then there exists a positive integer $t$ such that $a=17t$, $b=9t$, $L\sim 3\sigma+5l$ 
and $E\sim t(7\sigma+6l)$. 
By Claim \ref{Fn_coeff_claim} \eqref{Fn_coeff_claim2}, there exists a fiber $l$ such that 
$\coeff_lE\geq 17t/3$. Thus $20t/3\leq\coeff_\sigma E(\leq 7t)$. 
Let $P:=\sigma\cap l$. 
Assume that there exists  a point $P'\in\Delta\setminus\{P\}$. Then 
$l_{P'}\leq E$ by Claim \ref{Fn_coeff_claim} \eqref{Fn_coeff_claim1}, where $l_{P'}$ 
is the fiber passing through $P'$. However, 
$4=k\geq\deg(\Delta\cap l)+\deg(\Delta\cap l_{P'})=6$, a contradiction. 
Hence $2=\deg(\Delta\cap\sigma)=\mult_P(\Delta\cap\sigma)$ and 
$\mult_P(\Delta\cap l)=1$. 
On the other hand, 
\begin{eqnarray*}
2 & = & \deg(\Delta\cap l)-\mult_P(\Delta\cap l)=
\sum_{Q\in\Delta\setminus\{P\}}\mult_Q(\Delta\cap l)\\
 & \leq & \sum_{Q\in\Delta\setminus\{P\}}\mult_Q\Delta
=4-\mult_P\Delta\leq 4-\mult_P(\Delta\cap\sigma)=2.
\end{eqnarray*}
Thus $\mult_P\Delta=2$. Let $V\to X$ be the blowing up along $\Delta$ and 
$\Gamma\subset M$ be the strict transform of the 
exceptional curve of $V\to X$ over $P$. 
Then $\coeff_\Gamma E_M\geq\coeff_\Gamma {E'}_M^{\Delta, a-b}=3t$, where 
$E':=(20t/3)\sigma+(17t/3)l$. 
However, $\Gamma$ is a $(-1)$-curve, contrary to 
Lemma \ref{mds_lem} \eqref{mds_lem1}. Thus $(n, h, b/a, k)\neq(1, 5, 9/17, 4)$. 

\textbf{The case $n=2$:}
We have $6\leq h\leq 4a/b\leq 8$. 
If $h=8$, then $b/a=(16-k)/(30-k)=1/2$. Thus $(n, h, b/a, k)=(2, 8, 1/2, 2)$. 
If $h=7$, then $b/a=(14-k)/(24-k)\leq 11/21$ by Claim \ref{Fn_coeff_claim} 
\eqref{Fn_coeff_claim3}. Thus $(n, h, b/a, k)=(2, 7, 11/21, 3)$ or $(2, 7, 1/2, 4)$. 
If $h=6$, then $b/a=(12-k)/(18-k)\leq 11/18$ by Claim \ref{Fn_coeff_claim} 
\eqref{Fn_coeff_claim3}. Thus $(n, h, b/a, k)=(2, 6, 3/5, 3)$, $(2, 6, 4/7, 4)$, 
$(2, 6, 7/13, 5)$ or $(2, 6, 1/2, 6)$. 

\textbf{The case $n=3$:}
We have $9\leq h\leq 5a/b\leq 10$. 
If $h=10$, then $b/a=(17-k)/(33-k)=1/2$. Thus $(n, h, b/a, k)=(3, 10, 1/2, 1)$. 
If $h=9$, then $b/a=(15-k)/(27-k)\leq 14/27$ by Claim \ref{Fn_coeff_claim} 
\eqref{Fn_coeff_claim3}. Thus $(n, h, b/a, k)=(3, 9, 1/2, 3)$. 
\end{proof}

\begin{claim}\label{big_curve_claim}
Assume that there exists an irreducible curve $C$ on $X$ apart from $\sigma$, $l$ 
such that $(L_M\cdot\phi^{-1}_*C)=0$. 
Then $b/a=1/2$ and the triplet $(n, h, k)$ is one of $(0, 3, 6)$, $(1, 5, 5)$ or $(2, 6, 6)$. 
Moreover, $\Delta\subset C$ holds. 
If $(n, h, k)=(0, 3, 6)$, then $C\sim \sigma+l$. If $(n, h, k)=(1, 5, 5)$ or $(2, 6, 6)$, then 
$C=\sigma_\infty$. 
\end{claim}

\begin{proof}
Take $m\geq 1$ and $u\geq 0$ such that $C\in|m\sigma+(nm+u)l|$.
We note that $C$ is a rational curve by Proposition \ref{dP-basic_prop} 
and \cite[\S 4]{KoMo}. 
Let $(m_1,\dots,m_k)$ be the multiplicity sequence of $C$ with respects to $\phi$. 
By Proposition \ref{sing_prop}, we have:

\begin{enumerate}
\renewcommand{\theenumi}{\roman{enumi}}
\renewcommand{\labelenumi}{(\theenumi)}
\item\label{bmult_claim1}
$0\leq m_i\leq m$ for any $i$.
\item\label{bmult_claim2}
$hm+3u=\sum_{i=1}^km_i$.
\item\label{bmult_claim3}
$(m-1)(nm+2u-2)=\sum_{i=1}^km_i(m_i-1)$.
\end{enumerate}

By \eqref{bmult_claim1} and \eqref{bmult_claim2}, $k\geq h$ holds. 
Thus $b/a=1/2$ and the triplet $(n, h, k)$ is one of $(0, 3, 6)$, 
$(1, 5, 5)$ or $(2, 6, 6)$. Indeed, if $(n, h, k)=(0, 4, 4)$, then 
$u=0$ and $m=1$. This implies that $C=\sigma$, a contradiction. 

If $k=h$, then $u=0$ and $m_i=m$ for any $1\leq i\leq k$. Hence 
$(m-1)(nm-2)=km(m-1)$ by \eqref{bmult_claim3}. Thus $m=1$. This implies that 
$C=\sigma_\infty$. Since $\deg(\Delta\cap C)=k$, $\Delta\subset C$ 
holds. 

If $k\neq h$, then $(n, h, k)=(0, 3, 6)$. We can assume that $u\geq m\geq 1$. 
Assume that $m\geq 2$. Then $6m\leq 3(m+u)=\sum_{i=1}^6m_i\leq 6m$. 
Thus $m=u$ and $2u-2=6m$. This leads to a contradiction. 
Hence $m=1$. Then $u=1$ since $3+3u=\sum_{i=1}^6m_i\leq 6$. This implies that 
$C\in|\sigma+l|$. Since $\deg(\Delta\cap C)=6=k$, $\Delta\subset C$ holds. 
\end{proof}

\textbf{The case $(n, h, b/a, k)=(2,6,3/5,3)$:}
In this case, there exists a positive integer $t$ such that $a=5t$, $b=3t$, 
$L\sim 3\sigma+6l$ and $E\sim t(\sigma+2l)$. 
By Claim \ref{big_curve_claim}, there exists a fiber $l\leq E$. 
Moreover, such fiber is unique since $k=\deg(\Delta\cap l)$. Then $E=t(2\sigma+l)$. 
In particular, $\Delta\cap\sigma=\emptyset$ and $\Delta\subset l$ hold. 
If $t=1$, then this triplet is nothing but the type \textbf{$[$(5,3),2;1,2$]_1$}.

\textbf{The case $(n, h, b/a, k)=(2,6,4/7,4)$:}
In this case, there exists a positive integer $t$ such that $a=7t$, $b=4t$, 
$L\sim 3\sigma+6l$ and $E\sim 2t(\sigma+2l)$. 
By Claim \ref{big_curve_claim}, there exists a fiber $l\leq E$. 
Moreover, such fiber is unique since $4=k<2\deg(\Delta\cap l)=6$. 
Hence $E=2t(\sigma+2l)$. 
Moreover, $\deg(\Delta\cap\sigma)=(L\cdot\sigma)=0$. Thus 
$\Delta\cap\sigma=\emptyset$. 
Pick an arbitrary point $P\in\Delta$. Then $4t\mult_P(\Delta\cap l)=3t\mult_P\Delta$ 
by Example \ref{Sit1}. Thus $(\mult_P\Delta, \mult_P(\Delta\cap l))=(4, 3)$. 
If $t=1$, then this triplet is nothing but the type 
\textbf{$[$(7,4),2;2,4$]_1$}.

\textbf{The case $(n, h, b/a, k)=(2,6,7/13,5)$:}
In this case, there exists a positive integer $t$ such that $a=13t$, $b=7t$, 
$L\sim 3\sigma+6l$ and $E\sim 5t(\sigma+2l)$. 
By Claim \ref{big_curve_claim}, there exists a fiber $l\leq E$. 
Moreover, such fiber is unique since $4=k<2\deg(\Delta\cap l)=6$. 
Hence $E=5t(\sigma+2l)$. 
Moreover, $\deg(\Delta\cap\sigma)=(L\cdot\sigma)=0$. Thus 
$\Delta\cap\sigma=\emptyset$. 
Pick arbitrary point $P\in\Delta$. Then $10t\mult_P(\Delta\cap l)=6t\mult_P\Delta$ 
by Example \ref{Sit1}. Thus $(\mult_P\Delta, \mult_P(\Delta\cap l))=(5, 3)$. 
If $t=1$, then this triplet is nothing but the type 
\textbf{$[$(13,7),2;5,10$]_1$}.

\textbf{The case $(n, h, b/a, k)=(2,7,11/21,3)$:}
In this case, there exists a positive integer $t$ such that $a=21t$, $b=11t$, 
$L\sim 3\sigma+7l$ and $E\sim t(9\sigma+7l)$. 
By Claim \ref{Fn_coeff_claim} \eqref{Fn_coeff_claim2}, there exists a fiber $l\leq E$ 
and $\coeff_lE\geq 7t$. Thus $E=t(9\sigma+7l)$. 
Let $P:=\sigma\cap l$. Then $\mult_P(\Delta\cap\sigma)=\deg(\Delta\cap\sigma)=1$. 
Moreover, $10t\mult_P\Delta=9t+7t\mult_P(\Delta\cap l)$ by Example \ref{Sit2}. 
Thus $\mult_P\Delta=\mult_P(\Delta\cap l)=3$. 
If $t=1$, then this triplet is nothing but the type 
\textbf{$[$(21,11),2;9,7$]_1$}.

\textbf{The case $(n, h, b/a, k)=(0,3,1/2,6)$:}
In this case, there exists a positive integer $t$ such that $a=2t$, $b=t$, 
$L\sim 3\sigma+3l$ and $E\sim t(\sigma+l)$. 
Assume that there exists a component $C\leq E$ apart from $\sigma$, $l$. 
Then $C\in|\sigma+l|$ and $\Delta\subset C$ by Claim \ref{big_curve_claim}. 
Conversely, for any nonsingular $C\in|\sigma+l|$ and any subscheme $\Delta\subset C$ 
such that $\deg\Delta=6$, the triplet $(X, tC, \Delta)$ is a 
$(2t, t)$-fundamental triplet. Thus $E$ must be equal to $tC$ by 
Lemma \ref{triplet_lem} \eqref{triplet_lem6}. This contradicts to the assumption 
$\coeff(E/b)\not\subset\Z$. 
Therefore any component of $E$ is either a minimal section or a fiber. 
Assume that there exist distinct fibers $l_1$, $l_2\leq E$. 
Since $\deg(\Delta\cap l_i)=3$ for $i=1$, $2$, we have $\Delta\subset l_1\cup l_2$. 
Let $\sigma\leq E$ be a minimal section. 
Let $P_i:=\sigma\cap l_i$ $(i=1$, $2)$. 
Since $|\Delta|\cap\sigma\subset\{P_1, P_2\}$ and 
$\deg(\Delta\cap\sigma)=3$, we may assume that 
$\mult_{P_1}(\Delta\cap\sigma)\geq 2$. 
Then $\mult_{P_1}(\Delta\cap l_1)=1$, contradicts to the fact 
$\Delta\subset l_1\cup l_2$. 
Therefore $E=t\sigma+tl$ for some minimal section $\sigma$ and a fiber $l$. 
This contradicts to the assumption $\coeff(E/b)\not\subset\Z$. 
Thus there is no such fundamental triplet for the case  $(n, h, b/a, k)=(0,3,1/2,6)$.

\textbf{The case $(n, h, b/a, k)=(0,4,1/2,4)$:}
In this case, there exists a positive integer $t$ such that $a=2t$, $b=t$ and 
$E=t\sigma$. However, this contradicts to the assumption $\coeff(E/b)\not\subset\Z$. 
Thus there is no such fundamental triplet for the case  $(n, h, b/a, k)=(0,4,1/2,4)$.

\textbf{The case $(n, h, b/a, k)=(1,5,1/2,5)$:}
In this case, there exists a positive integer $t$ such that $a=2t$, $b=t$, 
$L\sim 3\sigma+5l$ and $E\sim t(\sigma+l)$. 
Assume that there exists a component $C\leq E$ apart from $\sigma$, $l$. 
Then $C=\sigma_\infty$ and $\Delta\subset C$ by Claim 
\ref{big_curve_claim}. 
Conversely, for any section of infinity $\sigma_\infty$ and 
any subscheme $\Delta\subset\sigma_\infty$ such that $\deg\Delta=5$, the 
triplet $(X, t\sigma_\infty, \Delta)$ is a $(2t, t)$-fundamental triplet. 
Thus $E$ must be equal to $t\sigma_\infty$ by 
Lemma \ref{triplet_lem} \eqref{triplet_lem6}. This contradicts to the assumption 
$\coeff(E/b)\not\subset\Z$. 
Therefore any component of $E$ is either $\sigma$ or a fiber. 
By Claim \ref{Fn_coeff_claim} \eqref{Fn_coeff_claim2}, $\coeff_lE\geq 2t/3$ for 
any fiber $l\leq E$. 
Thus $E=t(\sigma+l)$.
This contradicts to the assumption $\coeff(E/b)\not\subset\Z$. 
Thus there is no such fundamental triplet for the case  $(n, h, b/a, k)=(1,5,1/2,5)$.

\textbf{The case $(n, h, b/a, k)=(1,6,1/2,3)$:}
In this case, there exists a positive integer $t$ such that $a=2t$, $b=t$ and 
$E=t\sigma$. However, this contradicts to the assumption $\coeff(E/b)\not\subset\Z$. 
Thus there is no such fundamental triplet for the case  $(n, h, b/a, k)=(1,6,1/2,3)$.

\textbf{The case $(n, h, b/a, k)=(2,6,1/2,6)$:}
In this case, there exists a positive integer $t$ such that $a=2t$, $b=t$, 
$L\sim 3\sigma+6l$ and $E\sim t(\sigma+2l)$. 
Assume that there exists a component $C\leq E$ apart from $\sigma$, $l$. 
Then $C=\sigma_\infty$ and $\Delta\subset C$ by Claim 
\ref{big_curve_claim}. 
Conversely, for any section of infinity $\sigma_\infty$ and 
any subscheme $\Delta\subset\sigma_\infty$ such that $\deg\Delta=6$, the 
triplet $(X, t\sigma_\infty, \Delta)$ is a $(2t, t)$-fundamental triplet. 
Thus $E$ must be equal to $t\sigma_\infty$ by 
Lemma \ref{triplet_lem} \eqref{triplet_lem6}. This contradicts to the assumption 
$\coeff(E/b)\not\subset\Z$. 
Therefore we can write $E=t\sigma+\sum_{i=1}^jc_il_i$, where $c_i$ be positive integers 
with $\sum_{i=1}^j=2t$ and $l_i$ be distinct fibers.
Since $c_i<2t$, we have $j\geq 2$. 
We note that $c_i\geq t$ for any $i$ by Example \ref{Sit1} and the fact 
$\Delta\cap\sigma=\emptyset$.
Thus $E=t(\sigma+l_1+l_2)$.
This contradicts to the assumption $\coeff(E/b)\not\subset\Z$. 
Thus there is no such fundamental triplet for the case  $(n, h, b/a, k)=(2,6,1/2,6)$.

\textbf{The case $(n, h, b/a, k)=(2,7,1/2,4)$:}
In this case, there exists a positive integer $t$ such that $a=2t$, $b=t$, 
$L\sim 3\sigma+7l$ and $E\sim t(\sigma+l)$. 
Any component of $E$ is either $\sigma$ or a fiber by Claim \ref{big_curve_claim}. 
By Claim \ref{Fn_coeff_claim} \eqref{Fn_coeff_claim2}, $\coeff_lE\geq 2t/3$ for 
any fiber $l\leq E$. 
Thus $E=t(\sigma+l)$.
This contradicts to the assumption $\coeff(E/b)\not\subset\Z$. 
Thus there is no such fundamental triplet for the case  $(n, h, b/a, k)=(2,7,1/2,4)$.

\textbf{The case $(n, h, b/a, k)=(2,8,1/2,2)$:}
In this case, there exists a positive integer $t$ such that $a=2t$, $b=t$ and 
$E=t\sigma$. However, this contradicts to the assumption $\coeff(E/b)\not\subset\Z$. 
Thus there is no such fundamental triplet for the case  $(n, h, b/a, k)=(2,8,1/2,2)$.

\textbf{The case $(n, h, b/a, k)=(3,9,1/2,3)$:}
In this case, there exists a positive integer $t$ such that $a=2t$, $b=t$, 
$L\sim 3\sigma+9l$ and $E\sim t(\sigma+l)$. 
Any component of $E$ is either $\sigma$ or a fiber by Claim \ref{big_curve_claim}. 
By Claim \ref{Fn_coeff_claim} \eqref{Fn_coeff_claim2}, $\coeff_lE\geq 2t/3$ for 
any fiber $l\leq E$. 
Thus $E=t(\sigma+l)$.
This contradicts to the assumption $\coeff(E/b)\not\subset\Z$. 
Thus there is no such fundamental triplet for the case  $(n, h, b/a, k)=(3,9,1/2,3)$.

\textbf{The case $(n, h, b/a, k)=(3,10,1/2,1)$:}
In this case, there exists a positive integer $t$ such that $a=2t$, $b=t$ and 
$E=t\sigma$. However, this contradicts to the assumption $\coeff(E/b)\not\subset\Z$. 
Thus there is no such fundamental triplet for the case  $(n, h, b/a, k)=(3,10,1/2,1)$.

\subsection{The case $X=\mathbb{F}_n$ with $K_X+L$ non-big}\label{small_section}

We consider the case $X=\F_n$ such that the divisor $K_X+L$ is not big. 

Since $K_X+L$ is nef, not linearly trivial and not big, 
$L\sim 2\sigma+hl$ for some $h\geq n+3$. 
Since $L$ is nef and big, we have $h\geq\max\{2n, n+3\}$. 
We note that $E\sim 2(a-b)\sigma+((n+2)a-hb)l$.

\begin{claim}\label{sigma_empty_claim}
We have $h=2n$, $n\geq 3$ and $b/a\leq(n+1)/(2n-1)$ unless 
$n=1$, $h=4$ and $b/a=1/2$.
\end{claim}

\begin{proof}
We have the following: 
\begin{eqnarray*}
& & n(2(a-b))-((n+2)a-hb)
\geq
\begin{cases}
-a+2b\geq 0 & \text{if }n=1,\\
b>0 & \text{if }n=2,\\
(n-2)a>0 & \text{if }n\geq 3.
\end{cases}
\end{eqnarray*}
Therefore $\sigma\leq E$ unless $n=1$, $h=4$ and $b/a=1/2$.
If $\sigma\leq E$, then $h-2n=(L\cdot E)=0$. Thus $h=2n$. 
Hence $n\geq 3$ since $h\geq n+3$. 
Moreover, we have 
$a-b\leq\mult_PE\leq(n+2)a-2nb$
for any point $P\in\Delta$ by Proposition \ref{toric_prop}. 
Therefore we have $b/a\leq(n+1)/(2n-1)$. 
\end{proof}

We first consider the case $(n, h, b/a)=(1, 4, 1/2)$.
In this case, there exists a positive integer $t$ such that 
$a=2t$, $b=t$, $L\sim 2\sigma+4l$ and $E\sim 2t(\sigma+l)$. 
Then we have $kt=(L\cdot E)=8t$. Thus $k=8$ holds. 
Since $\coeff_\sigma E<a=2t$, 
there exists a component $C\leq E$ apart from $\sigma$, $l$. 
Define $m\geq 1$ and $u\geq 0$ such that $C\sim m\sigma+(m+u)l$. 
Let $(m_1,\dots,m_8)$ be the multiplicity sequence of $C$ with respects to $\phi$. 
Then we have $0\leq m_i\leq m$, $\sum_{i=1}^8m_i=(L\cdot C)=4m+2u$ 
and $(m-1)(m+2u-2)=2p_a(C)=\sum_{i=1}^8m_i(m_i-1)$
by Proposition \ref{sing_prop}. Then we have:
\[
m^2+(2u+1)m+2=\sum_{i=1}^8m_i^2\geq 8((4m+2u)/8)^2=(4m^2+4um+u^2)/2.
\]
Thus $2m^2-2m+u^2-4\leq 0$ holds. Therefore $(m, u)=(1, 0), (1, 1), (1, 2)$ or $(2, 0)$. 
However, if $m=1$, then $C$ is a section. Hence 
$2+2u=(C^2)+1\geq\deg(\Delta\cap C)=4+2u$ holds. This leads to a contradiction. 
Therefore $(m, u)=(2, 0)$. In this case, $C$ is nonsingular. 
Moreover, we have $\sum_{i=1}^8m_i=\sum_{i=1}^8m_i^2=8$. Thus $m_i=1$ for any $i$. 
In particular, $\Delta\subset C$ holds. Conversely, for any nonsingular 
$C\in|2\sigma+2l|$ and any subscheme $\Delta\subset C$ such that $\deg\Delta=8$, 
the triplet $(X, tC, \Delta)$ is a $(2t, t)$-fundamental triplet. 
Thus $E$ must be equal to $t\sigma_\infty$ by 
Lemma \ref{triplet_lem} \eqref{triplet_lem6}. This contradicts to the assumption 
$\coeff(E/b)\not\subset\Z$. 
Thus there is no such fundamental triplet for the case  $(n, h, b/a)=(1,4,1/2)$.

From now on, we consider the case $h=2n$, $n\geq 3$ and $b/a\leq(n+1)/(2n-1)$. 
Since $k(a-b)=(L\cdot E)=2(n+2)a-4nb$, we have 
\[
\frac{b}{a}=\frac{2n+4-k}{4n-k}.
\]
There exists a positive rational number $t$ with 
$t\cdot\gcd(4n-k, 2n+4-k)\in\Z$ such that 
$a=(4n-k)t$, $b=(2n+4-k)t$ (thus $a-b=2(n-2)t$), $L\sim 2\sigma+2nl$ and 
$E\sim(n-2)t(4\sigma+kl)$. 
Since $1/2\leq b/a\leq(n+1)/(2n-1)$, we have $2\leq k\leq 8$.

\begin{claim}\label{small_possible_claim}
\begin{enumerate}
\renewcommand{\theenumi}{\arabic{enumi}}
\renewcommand{\labelenumi}{(\theenumi)}
\item\label{small_possible_claim1}
We have the following: 
\begin{enumerate}
\renewcommand{\theenumii}{\alph{enumii}}
\renewcommand{\labelenumii}{(\theenumii)}
\item\label{small_possible_claim11}
$\coeff_\sigma E\geq(n-2)t\cdot(4n-k)/n$. 
\item\label{small_possible_claim12}
$\coeff_{l_P}E\geq\coeff_\sigma E-2(n-2)t\geq (n-2)t\cdot(2n-k)/n$ 
for any $P\in\Delta$, where $l_P$ is the fiber passing through $P$. 
\end{enumerate}
\item\label{small_possible_claim2}
Assume that there exists an irreducible curve $C$ on $X$ apart from $\sigma$, $l$ 
such that $(L_M\cdot\phi^{-1}_*C)=0$. 
Then $C=\sigma_\infty$ and $(n, k)=(3, 6), (3, 7), (3, 8)$ or $(4, 8)$ unless 
$C\sim\sigma+4l$ and $(n, k)=(3, 8)$.
\end{enumerate}
\end{claim}

\begin{proof}
\eqref{small_possible_claim1}
We have already seen that $\sigma\leq E$. 
Let $\sigma_M$ be the strict transform of $\sigma$ on $M$. 
Since $\Delta\cap\sigma=\emptyset$, the curve $\sigma_M$ satisfies that 
$(\sigma_M^2)=-n$. Thus we have the following: 
\[
(4n-k)(n-2)t=a(-K_M\cdot\sigma_M)=(E_M\cdot\sigma_M)\geq(-n)\coeff_\sigma E.
\]
Hence \eqref{small_possible_claim11} holds. 
On the other hand, for any $P\in\Delta$, 
$\mult_PE=\mult_P(E-(\coeff_\sigma E)\sigma)$ since 
$\Delta\cap\sigma=\emptyset$. 
Thus $\coeff_{l_P}E\geq\coeff_\sigma E-2(n-2)t$ by Proposition \ref{toric_prop}. 
Hence \eqref{small_possible_claim12} holds. 

\eqref{small_possible_claim2}
We define $m\geq 1$ and $u\geq 0$ such that $C\sim m\sigma+(nm+u)l$. 
Let $(m_1,\dots,m_k)$ be the multiplicity sequence of $C$ with respects to $\phi$. 
We have $0\leq m_i\leq m$, $\sum_{i=1}^km_i=(L\cdot C)=2nm+2u$ and 
$(m-1)(nm+2u-2)=2p_a(C)=\sum_{i=1}^km_i(m_i-1)$ by Proposition \ref{sing_prop}. 
Thus $2nm+2u\leq km$. Hence $k\geq 2n$. Therefore 
$(n, k)=(3, 6), (3, 7), (3, 8)$ or $(4, 8)$. 

We consider the case $(n, k)=(4, 8)$. 
Then $8m+2u=\sum_{i=1}^8m_i\leq 8m$. Thus $u=0$ and $m_i=m$ for any $i$. 
Moreover, since $8m(m-1)=(m-1)(4m-2)$, we have $m=1$. 
Therefore $C\sim\sigma+4l$. 

We consider the case $n=3$ (note that $6\leq k\leq 8$). 
We have $3m^2+(2u+1)m+2=\sum_{i=1}^km_i^2\geq((6m+2u)/k)^2k$. 
Thus we have: 
\begin{eqnarray*}
0 & \geq & (36-3k)m^2+(24u-(2u+1)k)m+(4u^2-2k)\\
&= & (36-3k)\left(m-\frac{(2u+1)k-24u}{2(36-3k)}\right)^2\\
 & + & 4u^2-2k-\frac{((2u+1)k-24u)^2}{4(36-3k)}.
\end{eqnarray*}
Hence $((2u+1)k-24u)^2/(4(36-3k))\geq 4u^2-2k$ holds.
Thus we have
$
4(k-12)(u^2+u-6)+k\geq 0.
$
If $u\geq 3$, then $4(k-12)(u^2+u-6)+k\leq 24(k-12)+k=25k-288<0$, 
a contradiction. Thus $u=0$, $1$ or $2$. 

\textbf{The case $u=2$:}
In this case, $0\geq(36-3k)m^2+(48-5k)m+(16-2k)=(3m+2)((12-k)m+8-k)$ holds. 
This leads to a contradiction. 

\textbf{The case $u=1$:}
In this case, $0\geq(36-3k)m^2+(24-3k)m+(4-2k)$ holds. Hence we have
\begin{eqnarray*}
m\leq\frac{-24+3k+\sqrt{192k-15k^2}}{2(36-3k)}=
\begin{cases}
\frac{-6+\sqrt{612}}{36}<1 & (k=6),\\
\frac{-3+\sqrt{609}}{30}<1 & (k=7),\\
1 & (k=8).
\end{cases}
\end{eqnarray*}
Therefore we have $m=1$ and $k=8$. 

\textbf{The case $u=0$:}
In this case, $0\geq(36-3k)m^2-km-2k$ holds. Hence we have
\begin{eqnarray*}
m\leq\frac{k+\sqrt{288k-23k^2}}{2(36-3k)}=
\begin{cases}
1 & (k=6),\\
\frac{7+\sqrt{889}}{30}<2 & (k=7),\\
\frac{8+\sqrt{832}}{24}<2 & (k=8).
\end{cases}
\end{eqnarray*}
Therefore we have $m=1$ and $k=6$, $7$, $8$. 
\end{proof}

\textbf{The case $k=2$:}
In this case, there exists a positive rational number $t$ 
with $t\cdot\gcd(4n-2, 2n+2)\in\Z$ such that 
$a=(4n-2)t$, $b=(2n+2)t$, $L\sim 2\sigma+2nl$ and 
$E\sim(n-2)t(4\sigma+2l)$. We note that 
\[
\gcd(4n-2, 2n+2)=
\begin{cases}
2 & (3\nmid n-2),\\
6 & (3\mid n-2).
\end{cases}
\]
By Claim \ref{small_possible_claim}, 
we have $E=(n-2)t(4\sigma+2l)$ for some fiber $l$. 
Since $\deg(\Delta\cap l)=2$, we have $\Delta\subset l$. 
If $3\nmid n-2$ and $t=1/2$, then this triplet is nothing but the type 
\textbf{$[$(2{\bi n}-1,{\bi n}+1),{\bi n};2({\bi n}-2),{\bi n}-2$]_1$}. 
If $3\mid n-2$ (set $3m:=n-2$) and $t=1/6$, then this triplet is nothing but the type 
\textbf{$[$(2{\bi m}+1,{\bi m}+1),3{\bi m}+2;2{\bi m},{\bi m}$]_1$}.

\textbf{The case $k=3$:}
In this case, there exists a positive rational number $t$ 
with $t\cdot\gcd(4n-3, 2n+1)\in\Z$ such that 
$a=(4n-3)t$, $b=(2n+1)t$, $L\sim 2\sigma+2nl$ and 
$E\sim(n-2)t(4\sigma+3l)$. We note that 
\[
\gcd(4n-3, 2n+1)=
\begin{cases}
1 & (5\nmid n-2),\\
5 & (5\mid n-2).
\end{cases}
\]
By Claim \ref{small_possible_claim}, 
we have $E=(n-2)t(4\sigma+3l)$ for some fiber $l$. 
For any point $P\in\Delta$, we have 
$3(n-2)t\mult_P(\Delta\cap l)=2(n-2)t\mult_P\Delta$ by Example \ref{Sit1}. 
Thus $\mult_P\Delta=3$ and $\mult_P(\Delta\cap l)=2$. 
If $5\nmid n-2$ and $t=1$, then this triplet is nothing but the type 
\textbf{$[$(4{\bi n}-3,2{\bi n}+1),{\bi n};4({\bi n}-2),3({\bi n}-2)$]_1$}. 
If $5\mid n-2$ (set $5m:=n-2$) and $t=1/5$, then this triplet is nothing but the type 
\textbf{$[$(4{\bi m}+1,2{\bi m}+1),5{\bi m}+2;4{\bi m},3{\bi m}$]_1$}.

\textbf{The case $k=4$:}
In this case, there exists a positive rational number $t$ 
with $t\cdot\gcd(4n-4, 2n)\in\Z$ such that 
$a=(4n-4)t$, $b=2nt$, $L\sim 2\sigma+2nl$ and 
$E\sim(n-2)t(4\sigma+4l)$. We note that 
\[
\gcd(4n-4, 2n)=
\begin{cases}
2 & (2\nmid n),\\
4 & (2\mid n).
\end{cases}
\]
By Claim \ref{small_possible_claim}, 
we have $E=4(n-2)t\sigma+\sum_{i=1}^jc_il_i$ such that $c_i\geq 2(n-2)t$, 
$\sum_{i=1}^jc_i=4(n-2)t$ and $l_i$ are distinct fibers. 
In particular, $j\leq 2$ holds. 

We consider the case $j=2$. Then $E=2(n-2)t(2\sigma+l_1+l_2)$. 
Since $\sum_{i=1}^2\deg(\Delta\cap l_i)=2+2=4=k$, we have 
$\Delta\subset l_1\cup l_2$. 
If $2\nmid n$ and $t=1/2$, then this triplet is nothing but the type 
\textbf{$[$(2{\bi n}-2,{\bi n}),{\bi n};2({\bi n}-2),2({\bi n}-2)$]_{11}$}. 
If $2\mid n$ (set $2m:=n-2$) and $t=1/4$, then this triplet is nothing but the type 
\textbf{$[$(2{\bi m}+1,{\bi m}+1),2{\bi m}+2;2{\bi m},2{\bi m}$]_{11}$}. 

We consider the case $j=1$. Then $E=4(n-2)t(\sigma+l)$. 
For any point $P\in\Delta$, we have 
$4(n-2)t\mult_P(\Delta\cap l)=2(n-2)t\mult_P\Delta$ by Example \ref{Sit1}. 
Thus $\mult_P\Delta=2\mult_P(\Delta\cap l)$. 
Assume that $|\Delta|=\{P\}$. Then $(\mult_P\Delta, \mult_P(\Delta\cap l))=(4, 2)$. 
If $2\nmid n$ and $t=1/2$, then this triplet is nothing but the type 
\textbf{$[$(2{\bi n}-2,{\bi n}),{\bi n};2({\bi n}-2),2({\bi n}-2)$]_1$(1)}. 
If $2\mid n$ (set $2m:=n-2$) and $t=1/4$, then this triplet is nothing but the type 
\textbf{$[$(2{\bi m}+1,{\bi m}+1),2{\bi m}+2;2{\bi m},2{\bi m}$]_1$(1)}. 
Assume that $|\Delta|=\{P_1, P_2\}$. 
Then $(\mult_{P_i}\Delta, \mult_{P_i}(\Delta\cap l))=(2, 1)$ for $i=1$, $2$. 
If $2\nmid n$ and $t=1/2$, then this triplet is nothing but the type 
\textbf{$[$(2{\bi n}-2,{\bi n}),{\bi n};2({\bi n}-2),2({\bi n}-2)$]_1$(2)}. 
If $2\mid n$ (set $2m:=n-2$) and $t=1/4$, then this triplet is nothing but the type 
\textbf{$[$(2{\bi m}+1,{\bi m}+1),2{\bi m}+2;2{\bi m},2{\bi m}$]_1$(2)}.

\textbf{The case $k=5$:}
In this case, there exists a positive rational number $t$ 
with $t\cdot\gcd(4n-5, 2n-1)\in\Z$ such that 
$a=(4n-5)t$, $b=(2n-1)t$, $L\sim 2\sigma+2nl$ and 
$E\sim(n-2)t(4\sigma+5l)$. We note that 
\[
\gcd(4n-5, 2n-1)=
\begin{cases}
1 & (3\nmid n-2),\\
3 & (3\mid n-2).
\end{cases}
\]
By Claim \ref{small_possible_claim}, 
we have $E=4(n-2)t\sigma+\sum_{i=1}^jc_il_i$ such that $c_i\geq 2(n-2)t$, 
$\sum_{i=1}^jc_i=5(n-2)t$ and $l_i$ are distinct fibers. 
In particular, $j\leq 2$ holds. 

We consider the case $j=2$. We can assume that $c_1\geq c_2$. 
For any point $P_i\in\Delta\cap l_i$, we have 
$c_i\mult_{P_i}(\Delta\cap l_i)=2(n-2)t\mult_{P_i}\Delta$. 
Since $5(n-2)t/2\leq c_1\leq 3(n-2)t$ and $\mult_{P_1}(\Delta\cap l_1)\leq 2$, 
we have $(c_1, c_2)=(3(n-2)t, 2(n-2)t)$, 
$(\mult_{P_1}\Delta, \mult_{P_1}(\Delta\cap l_1))=(3, 2)$, and 
$(\mult_{P_2}\Delta, \mult_{P_2}(\Delta\cap l_2))=(1, 1)$ or $(2, 2)$. 
If $3\nmid n-2$ and $t=1$, then this triplet is nothing but the type 
\textbf{$[$(4{\bi n}-5,2{\bi n}-1),{\bi n};4({\bi n}-2),5({\bi n}-2)$]_{32}$}. 
If $3\mid n-2$ (set $3m:=n-2$) and $t=1/3$, then this triplet is nothing but the type 
\textbf{$[$(4{\bi m}+1,2{\bi m}+1),3{\bi m}+2;4{\bi m},5{\bi m}$]_{32}$}. 

We consider the case $j=1$. Then $E=(n-2)t(4\sigma+5l)$. 
For any $P\in\Delta$, we have $5(n-2)t\mult_P(\Delta\cap l)=2(n-2)t\mult_P\Delta$ 
by Example \ref{Sit1}. Thus $|\Delta|=\{P\}$, $\mult_P\Delta=5$ and 
$\mult_P(\Delta\cap l)=2$. In this case, the maximum of the coefficients of $E_M$ 
is equal to $6(n-2)t$. This value must be smaller than $a=(4n-5)t$. Thus $n=3$. 
If $t=1$, then this triplet is nothing but the 
type \textbf{$[$(7,5),3;4,5$]_1$}.

\textbf{The case $k=6$:}
In this case, there exists a positive rational number $t$ 
with $2t\in\Z$ such that 
$a=(4n-6)t$, $b=(2n-2)t$, $L\sim 2\sigma+2nl$ and 
$E\sim(n-2)t(4\sigma+6l)$. 
Assume that there exists a component of $E$ apart from $\sigma$, $l$. 
Then such component must be $\sigma_\infty$ and $n=3$ by 
Claim \ref{small_possible_claim}. 
Since $\deg(\Delta\cap\sigma_\infty)=(L\cdot\sigma_\infty)=6=k$, 
we have $\Delta\subset\sigma_\infty$. Conversely, for $n=3$ and for any 
section at infinity $\sigma_\infty$, any subscheme $\Delta\subset\sigma_\infty$ 
with $\deg\Delta=6$, the triplet $(X, 2t(\sigma+\sigma_\infty), \Delta)$ is a 
$(6t, 4t)$-fundamental triplet. 
Thus $E$ must be equal to $2t(\sigma+\sigma_\infty)$ by 
Lemma \ref{triplet_lem} \eqref{triplet_lem6}. 
If $t=1/2$, then this triplet is nothing but the 
type \textbf{$[$(3,2),3;2,3$]_{1\infty}$}. 
We assume that $E=4(n-2)t\sigma+\sum_{i=1}^jc_il_i$ such that $c_i\geq 2(n-2)t$, 
$\sum_{i=1}^jc_i=6(n-2)t$ and $l_i$ are distinct fibers. 
Thus $2\leq j\leq 3$ since $6(n-2)t\geq a=(4n-6)t$. 

We consider the case $j=3$. 
Then $E=2(n-2)t(2\sigma+l_1+l_2+l_3)$. 
Since $\sum_{i=1}^3\deg(\Delta\cap l_i)=6=k$, we have 
$\Delta\subset l_1\cup l_2\cup l_3$. 
If $t=1/2$, then this triplet is nothing but the type 
\textbf{$[$(2{\bi n}-3,{\bi n}-1),{\bi n};2({\bi n}-2),3({\bi n}-2)$]_{111}$}. 

We consider the case $j=2$. 
We can assume that $4(n-2)t\geq c_1\geq 3(n-2)t\geq c_2\geq 2(n-2)t$. 
Note that $\deg(\Delta\cap l_i)=2$ for $i=1$, $2$. 
For any $P_i\in\Delta\cap l_i$, we have $c_i\mult_{P_i}(\Delta\cap l_i)
=2(n-2)t\mult_{P_i}\Delta$ by Example \ref{Sit1}. 
Thus $(c_1, c_2)=(4(n-2)t, 2(n-2)t)$ or $(3(n-2)t, 3(n-2)t)$. 
Assume that $(c_1, c_2)=(4(n-2)t, 2(n-2)t)$ and $|\Delta|\cap l_1=\{P_1\}$. 
Then $(\mult_{P_1}\Delta, \mult_{P_1}(\Delta\cap l_1))=(4, 2)$. 
If $t=1/2$, then this triplet is nothing but the type 
\textbf{$[$(2{\bi n}-3,{\bi n}-1),{\bi n};2({\bi n}-2),3({\bi n}-2)$]_{21}$(1)}. 
Assume that $(c_1, c_2)=(4(n-2)t, 2(n-2)t)$ and $|\Delta|\cap l_1=\{P_{11}, P_{12}\}$. 
Then $(\mult_{P_{1q}}\Delta, \mult_{P_{1q}}(\Delta\cap l_1))=(2, 1)$ for $q=1$, $2$. 
If $t=1/2$, then this triplet is nothing but the type 
\textbf{$[$(2{\bi n}-3,{\bi n}-1),{\bi n};2({\bi n}-2),3({\bi n}-2)$]_{21}$(2)}. 
Assume that $(c_1, c_2)=(3(n-2)t, 3(n-2)t)$. 
Then $|\Delta|\cap l_i=\{P_i\}$, $\mult_{P_i}\Delta=3$ and 
$\mult_{P_i}(\Delta\cap l_i)=2$ for $i=1$, $2$. 
If $2\nmid n$ and $t=1$, then this triplet is nothing but the type 
\textbf{$[$(4{\bi n}-6,2{\bi n}-2),{\bi n};4({\bi n}-2),6({\bi n}-2)$]_{11}$}. 
If $2\mid n$ (set $2m:=n-2$) and $t=1/2$, then this triplet is nothing but the type 
\textbf{$[$(4{\bi m}+1,2{\bi m}+1),2{\bi m}+2;4{\bi m},6{\bi m}$]_{11}$}.

\textbf{The case $k=7$:}
In this case, there exists a positive integer $t$ such that 
$a=(4n-7)t$, $b=(2n-3)t$, $L\sim 2\sigma+2nl$ and 
$E\sim(n-2)t(4\sigma+7l)$. 

Assume that there exists a component of $E$ apart from $\sigma$, $l$. 
Then such component must be equal to $\sigma_\infty$ and $n=3$ by 
Claim \ref{small_possible_claim}. We note that $\deg(\Delta\cap\sigma_\infty)=6$. 
The curve $\sigma_\infty$ is the only component 
of $E$ apart from $\sigma$, $l$. Indeed, if there exists another component 
$\sigma'_\infty\leq E$, then $(\sigma_\infty\cdot\sigma'_\infty)\geq 6+6-7=5$ 
by Corollary \ref{twocurve_cor}. This leads to a contradiction. 
Thus $E=(4t-d)\sigma+d\sigma_\infty+\sum_{i=1}^jc_il_i$, where 
$0<d\leq 2t$, $j\geq 0$, $l_i$ are distinct fibers and $3d+\sum_{i=1}^jc_i=7t$ 
by Definition \ref{fund_dfn} \eqref{fund_dfn72}.
Let $\sigma_{\infty,M}$ be the strict transform of $\sigma_\infty$ on $M$. Then 
$(\sigma_{\infty,M}^2)=-3$. Thus 
$-5t=a(-K_M\cdot\sigma_{\infty,M})=(E_M\cdot\sigma_{\infty,M})\geq
(-3)\coeff_{\sigma_\infty}E=-3d$. Hence $d\geq 5t/3$. 
Since $\deg(\Delta\cap l_i)=2$, $k\geq\deg(\Delta\cap\sigma_\infty)+j=6+j$. 
Thus $j=1$ (if $j=0$, then $d=7t/3$, a contradiction). 
Set $P:=\sigma_\infty\cap l_1$. We consider the case such that there exists a point 
$Q\in\Delta\cap(l_1\setminus\sigma_\infty)$. Then $\mult_Q\Delta=1$ since 
$k\geq\deg(\Delta\cap\sigma_\infty)+\mult_Q\Delta=6+\mult_Q\Delta$. 
Thus $c_1=2t$ and $d=5t/3$ by Example \ref{Sit1}. Moreover, 
$2t\mult_P\Delta=(5t/3)\mult_P(\Delta\cap\sigma_\infty)+2t$ by Example \ref{Sit2}. 
Thus $\mult_P\Delta=\mult_P(\Delta\cap\sigma_\infty)=6$. 
If $t=3$, then this triplet is nothing but the type 
\textbf{$[$(15,9),3;12,21$]_{5\infty 1}$}. 
Now we consider the case $|\Delta|\cap l_1=\{P\}$. Then $\mult_P(\Delta\cap l_1)=2$. 
Thus $2t\mult_P\Delta=d+2(7t-3d)=14t-5d$. 
Since $d\geq 5t/3$, we have $\mult_P\Delta=2$ and $d=2t$. 
Then for any $Q\in \Delta\cap(\sigma_\infty\setminus\{P\})$, 
$\mult_Q\Delta=\mult_Q(\Delta\cap\sigma_\infty)$. 
If $t=1$, then this triplet is nothing but the type 
\textbf{$[$(5,3),3;4,7$]_{2\infty 1}$}. 

Assume that any component of $E$ is either $\sigma$ or a fiber. 
Then $E=4(n-2)t\sigma+\sum_{i=1}^jc_il_i$ such that $\sum_{i=1}^jc_i=7(n-2)t$ and 
$l_i$ are distinct fibers. 
Since $2(n-2)t\leq c_i<(4n-7)t$, we have $j=2$ or $3$. 
We have $\deg(\Delta\cap l_i)=2$ for $i=1$, $2$. 
Moreover, for any $P_i\in\Delta\cap l_i$, 
$c_i\mult_{P_i}(\Delta\cap l_i)=2(n-2)t\mult_{P_i}\Delta$. 
Since $\mult_{P_i}(\Delta\cap l_i)\leq 2$ and $5(n-2)t\geq(4n-7)t$, 
we have $c_i=2(n-2)t$, $3(n-2)t$ or $4(n-2)t$. 

We consider the case $j=3$. 
Then we can assume that $(c_1, c_2, c_3)=(3(n-2)t, 2(n-2)t, 2(n-2)t)$. 
Moreover, $|\Delta|\cap l_1=\{P_1\}$, $\mult_{P_1}\Delta=3$ and 
$\mult_{P_q}\Delta=\mult_{P_q}(\Delta\cap l_q)$ for any $q=2$, $3$. 
If $t=1$, then this triplet is nothing but 
the type \textbf{$[$(4{\bi n}-7,2{\bi n}-3),{\bi n};4({\bi n}-2),7({\bi n}-2)$]_{322}$}. 

We consider the case $j=2$. Then we can assume that $(c_1, c_2)=(4(n-2)t, 3(n-2)t)$. 
Moreover, $|\Delta|\cap l_2=\{P_2\}$ and $\mult_{P_2}\Delta=3$. 
We consider the case $|\Delta|\cap l_1=\{P_1\}$. Then $\mult_{P_1}\Delta=4$. 
If $t=1$, then this triplet is nothing but 
the type \textbf{$[$(4{\bi n}-7,2{\bi n}-3),{\bi n};4({\bi n}-2),7({\bi n}-2)$]_{43}$(1)}. 
We consider the case $|\Delta|\cap l_1=\{P_{11}, P_{12}\}$. Then 
$\mult_{P_{1q}}\Delta=2$ for $q=1$, $2$. 
If $t=1$, then this triplet is nothing but 
the type \textbf{$[$(4{\bi n}-7,2{\bi n}-3),{\bi n};4({\bi n}-2),7({\bi n}-2)$]_{43}$(2)}.

\textbf{The case $k=8$:}
In this case, there exists a positive integer $t$ such that 
$a=2t$, $b=t$, $L\sim 2\sigma+2nl$ and 
$E\sim(n-2)t(2\sigma+4l)$. Since $\coeff_\sigma E<a=2t$, there exists a component 
$C\leq E$ apart from $\sigma$, $l$. By Claim \ref{small_possible_claim}, 
one of the following holds: $n=4$ and $C=\sigma_\infty$, $n=3$ and $C\sim\sigma+4l$, 
or $n=3$ and $C=\sigma_\infty$. 
Assume that $n=4$ and $C=\sigma_\infty$. Then $\Delta\subset\sigma_\infty$ 
since $\deg(\Delta\cap\sigma_\infty)=8$. 
Conversely, for $n=4$, for any section at infinity $\sigma_\infty$ and any subscheme 
$\Delta\subset\sigma_\infty$ such that $\deg\Delta=8$, the triplet 
$(X, t\sigma+t\sigma_\infty, \Delta)$ is a $(2t, t)$-fundamental triplet. 
Thus $E=t\sigma+t\sigma_\infty$ by Lemma \ref{triplet_lem} \eqref{triplet_lem6}. 
This contradicts to the assumption $\coeff(E/b)\not\subset\Z$. 
Assume that $n=3$ and $C\sim\sigma+4l$. 
Then $\Delta\subset C$ 
since $\deg(\Delta\cap C)=8$. Conversely, for $n=3$, 
for an irreducible $C\sim\sigma+4l$ and any subscheme 
$\Delta\subset C$ such that $\deg\Delta=8$ and $\Delta\cap\sigma=\emptyset$, 
the triplet 
$(X, t\sigma+tC, \Delta)$ is a $(2t, t)$-fundamental triplet (since $C$ and $\sigma$ 
intersect transversally). 
Thus $E=t\sigma+tC$ by Lemma \ref{triplet_lem} \eqref{triplet_lem6}. 
This contradicts to the assumption $\coeff(E/b)\not\subset\Z$. 
Thus $n=3$ and $C=\sigma_\infty$. We note that $\deg(\Delta\cap\sigma_\infty)=6$. 
The curve $\sigma_\infty$ is the only component of $E$ apart from $\sigma$, $l$. 
Indeed, if there exists another component 
$\sigma'_\infty\leq E$, then $(\sigma_\infty\cdot\sigma'_\infty)\geq 6+6-8=4$ 
by Corollary \ref{twocurve_cor}. This leads to a contradiction. 
Thus $E=(2t-d)\sigma+d\sigma_\infty+\sum_{i=1}^jc_il_i$, where 
$0<d\leq t$, $j\geq 0$, $l_i$ are distinct fibers and $3d+\sum_{i=1}^jc_i=4t$ 
by Definition \ref{fund_dfn} \eqref{fund_dfn72}.
Let $\sigma_{\infty,M}$ be the strict transform of $\sigma_\infty$ on $M$. Then 
$(\sigma_{\infty,M}^2)=-3$. Thus 
$-2t=a(-K_M\cdot\sigma_{\infty,M})=(E_M\cdot\sigma_{\infty,M})\geq
(-3)\coeff_{\sigma_\infty}E=-3d$. Hence $d\geq 2t/3$. 
Since $\deg(\Delta\cap l_i)=2$, $k\geq\deg(\Delta\cap\sigma_\infty)+j=6+j$. 
Thus $j=1$ or $2$ (if $j=0$, then $d=4t/3$, a contradiction). 

We consider the case $j=1$. Then $c_1=4t-3d$. If $d<t$, then 
$|\Delta|\cap\sigma_\infty=\{P\}$, where $P:=\sigma_\infty\cap l_1$. 
Thus $\mult_P(\Delta\cap\sigma_\infty)=6$. Thus $\mult_P(\Delta\cap l_1)=1$. 
Hence there exists a point $Q\in\Delta\cap(l_1\setminus\{P\})$. 
Since $\mult_Q(\Delta\cap l_2)=1$ and 
$\mult_Q\Delta\leq\deg\Delta-\deg(\Delta\cap\sigma_\infty)=2$, $4t-3d=t$ or $2t$ 
holds. By the assumption $d<t$, $d=2t/3$. However, in this case, $c_1=2t$ holds, 
a contradiction. Thus $d=t$, that is,  
$E=t(\sigma+\sigma_\infty+l_1)$. 
This contradicts to the assumption $\coeff(E/b)\not\subset\Z$. 

We consider the case $j=2$. Let $P_i:=\sigma_\infty\cap l_i$ for $i=1$, $2$. 
Assume that $d=t$. Then $c_1+c_2=t$ (thus $c_i<t$), $|\Delta|\cap l_i=\{P_i\}$ and 
$\mult_{P_i}(\Delta\cap l_i)=2$ holds. 
Hence $t\mult_{P_i}\Delta=2c_i+t$. 
Taking the sum, we have 
$4t\leq t(\mult_{P_1}\Delta+\mult_{P_2}\Delta)=4t$. Thus $\mult_{P_i}\Delta=2$ 
and $c_i=1/2$ for $i=1$, $2$. 
If $t=2$, then this triplet is nothing but 
the type \textbf{$[$(4,2),3;4,8$]_{2\infty 11}$}. 
Now assume that $d<t$. Then $|\Delta|\cap\sigma_\infty\subset\{P_1, P_2\}$. 
We can assume that $\mult_{P_1}(\Delta\cap\sigma_\infty)\geq 3$. 
Then $t\mult_{P_1}\Delta=c_1+d\mult_{P_1}(\Delta\cap\sigma_\infty)$. 
Since there exists a point $Q_1\in\Delta\cap l_1\setminus\{P_1\}$ such that 
$\mult_{Q_1}(\Delta\cap l_1)=1$, $c_1=t\mult_{Q_1}\Delta$. 
Since $c_1<2t$, we have $c_1=t$, $\mult_{Q_1}\Delta=1$ and 
$t(\mult_{P_1}\Delta-1)=d\mult_{P_1}(\Delta\cap\sigma_\infty)$. 
Thus 
\[
t>d=t\frac{\mult_{P_1}\Delta-1}{\mult_{P_1}(\Delta\cap\sigma_\infty)}.
\]
Since $\mult_{P_1}\Delta\geq\mult_{P_1}(\Delta\cap\sigma_\infty)$, we have 
$\mult_{P_1}\Delta=\mult_{P_1}(\Delta\cap\sigma_\infty)$ and 
$(d, \mult_{P_1}\Delta)=(2t/3, 3), (3t/4, 4), (4t/5, 5)$ or $(5t/6, 6)$. 
If $(d, \mult_{P_1}\Delta)=(2t/3, 3)$, then $\mult_{P_2}(\Delta\cap\sigma_\infty)=3$. 
By using the same argument, we have $c_2=t$ and 
$\mult_{P_2}\Delta=\mult_{P_2}(\Delta\cap\sigma_\infty)=3$. 
If $t=3$, then this triplet is nothing but 
the type \textbf{$[$(6,3),3;6,12$]_{2\infty 11}$}. 
If $(d, \mult_{P_1}\Delta)=(3t/4, 4)$, then $\mult_{P_2}(\Delta\cap l_2)=1$ since 
$\mult_{P_2}(\Delta\cap\sigma_\infty)=2$. 
Since $c_2=4t-3d-c_1=3t/4$, we have $t\mult_{P_2}\Delta=3t/4+2(3t/4)$. 
This leads to a contradiction since $\mult_{P_2}\Delta\in\Z$. 
If $(d, \mult_{P_1}\Delta)=(4t/5, 5)$, then $\mult_{P_2}(\Delta\cap\sigma_\infty)=1$ and 
$c_2=4t-3d-c_1=3t/5$. 
Thus $t\mult_{P_2}\Delta=4t/5+(3t/5)\mult_{P_2}(\Delta\cap l_2)$. 
Hence $\mult_{P_2}\Delta=\mult_{P_2}(\Delta\cap l_2)=2$. 
If $t=5$, then this triplet is nothing but 
the type \textbf{$[$(10,5),3;10,20$]_{4\infty 53}$}. 
If $(d, \mult_{P_1}\Delta)=(5t/6, 6)$, then $\mult_{P_2}(\Delta\cap\sigma_\infty)=0$ 
and $c_2=t/2$. This contradicts that $\deg(\Delta\cap l_2)=2$.

Consequently, we have completed the proof of 
Theorems \ref{mainthm1} and \ref{mainthm2}.

\section{Some structure properties}\label{str_section}

\subsection{From $(a,b)$-fundamental triplets to $(a,b)$-basic pairs}\label{corr_section}

In this section, we see that normalized $(a, b)$-fundamental triplets are uniquely 
determined by the associated $(a, b)$-basic pairs. The proof is essentially same as 
\cite[Theorem 4.9]{N}. We recall the following proposition in order to prove 
Theorem \ref{kaburi_thm}.

\begin{proposition}[{\cite[Proposition 4.10]{N}}]\label{invol_prop}
Let $f\colon Y\to T$ be a proper surjective morphism from a nonsingular surface $Y$ 
to a nonsingular curve $T$ such that a general fiber is isomorphic to $\pr^1$. 
Let $E_1$ and $E_2$ be two sections of $f$ such that $E_1\cap E_2=\emptyset$ 
and $K_Y+E_1+E_2$ is $f$-numerically trivial. 
Assume that $f_*\sO_{E_1}(E_1)\simeq f_*\sO_{E_2}(E_2)$. Let $t\in\Z_{\geq 0}$. 
Let $\Gamma_{i1}+\Gamma_{i2}$ be a fiber of $f$ for any $1\leq i\leq t$, 
let $\hat{Y}\to Y$ be the blowing up along the all the intersection points 
$\Gamma_{i1}\cap\Gamma_{i2}$, and let $\hat{E}_1$, $\hat{E}_2$, $\hat{\Gamma}_{i1}$, 
$\hat{\Gamma}_{i2}$ be the strict transform of $E_1$, $E_2$, $\Gamma_{i1}$, 
$\Gamma_{i2}$, respectively. 
Then there is an involution $\iota$ of $\hat{Y}$ over $T$ such that 
$\iota(\hat{E}_1)=\hat{E}_2$ and $\iota(\hat{\Gamma}_{i1})=\hat{\Gamma}_{i2}$ for 
all $1\leq i\leq t$.
\end{proposition}

\begin{proof}
By \cite[Proposition 4.10]{N}, there exists an involution $\iota'$ of $Y$ over $T$ such 
that $\iota'(E_1)=E_2$. The involution $\iota'$ fixes the intersection point 
$\Gamma_{i1}\cap\Gamma_{i2}$ for any $i$. Thus $\iota'$ lifts to an involution $\iota$ 
of $\hat{Y}$. Then the assertion follows. 
\end{proof}

\begin{thm}\label{kaburi_thm}
Let $a$, $b$ be positive integers with $1/2\leq b/a<1$, $b>1$, 
and let $(X, E, \Delta)$ be a normalized $(a, b)$-fundamental triplet. 
Then the isomorphism class of $(X, E, \Delta)$ depends only on the 
isomorphism class of the associated normalized $(a, b)$-basic pair $(M, E_M)$.
\end{thm}

\begin{proof}
We may assume that $K_X+L$ is not big. Let $(X', E', \Delta')$ be another 
normalized $(a, b)$-fundamental triplet whose associated $(a, b)$-basic pair 
is $(M, E_M)$. Let $\phi\colon M\to X$, $\phi'\colon M\to X'$ be the elimination 
of $\Delta$, $\Delta'$, respectively. Then $\pi\circ\phi=\pi'\circ\phi'$ holds, where 
$\pi\colon X\to\pr^1$ and $\pi'\colon X'\to\pr^1$ are the fibrations. 
Then the types of $(X, E, \Delta)$ and $(X', E', \Delta')$ are same. 
Moreover, 
$\phi\simeq\phi'$ over $\pr^1$ excepts for the case that the type is one of 
\textbf{$[$(3,2),3;2,3$]_{1\infty}$}, \textbf{$[$(5,3),3;4,7$]_{2\infty 1}$} or 
\textbf{$[$(4,2),3;4,8$]_{2\infty 11}$} by \cite[Lemma 4.5]{N}. 

Assume that the type is \textbf{$[$(3,2),3;2,3$]_{1\infty}$}. 
Then $E_M=\sigma_M+\sigma_{\infty,M}$, where $\sigma_M$, $\sigma_{\infty,M}$ is 
the strict transform of $\sigma$, $\sigma_\infty$, respectively. 
Then there exists an involution $\iota$ of $M$ over $\pr^1$ such that 
$\iota(\sigma_M)=\sigma_{\infty,M}$ by Proposition \ref{invol_prop}. 
Thus $\phi'\simeq\phi\circ\iota$ and $(X, E, \Delta)\simeq(X', E', \Delta')$. 

Assume that the type is \textbf{$[$(5,3),3;4,7$]_{2\infty 1}$}. 
Let $P:=\sigma_\infty\cap l$, where $\sigma_\infty$, $l\leq E$. 
Let $\Gamma_1$, $\Gamma_2$ be the exceptional curve on $M$ over $P$ with 
$(\Gamma_1^2)=-2$, $(\Gamma_2^2)=-1$. 
Then $E_M=2\sigma_M+l_M+\Gamma_1+2\sigma_{\infty,M}$, where 
$l_M$, $\sigma_M$, $\sigma_{\infty,M}$ is the strict transform of $l$, $\sigma$, 
$\sigma_{\infty}$, respectively. 
Then there exists an involution $\iota$ of $M$ over $\pr^1$ such that 
$\iota(\sigma_M)=\sigma_{\infty,M}$ and $\iota(l_M)=\Gamma_1$ 
by Proposition \ref{invol_prop} (apply Proposition \ref{invol_prop} for the surface $M'$, 
where $M\to M'$ is obtained by contracting $\Gamma_2$). 
Thus $\phi'\simeq\phi\circ\iota$ and $(X, E, \Delta)\simeq(X', E', \Delta')$. 

Assume that the type is \textbf{$[$(4,2),3;4,8$]_{2\infty 11}$}. 
Let $P_i:=\sigma_\infty\cap l_i$, where $\sigma_\infty$, $l_i\leq E$ 
$(i=1$, $2)$. 
Let $\Gamma_{i1}$, $\Gamma_{i2}$ be the exceptional curves on $M$ over $P_i$ with 
$(\Gamma_{i1}^2)=-2$, $(\Gamma_{i2}^2)=-1$. 
Then $E_M=2\sigma_M+l_{1M}+l_{2M}+\Gamma_{11}+\Gamma_{21}+2\sigma_{\infty,M}$, 
where $l_{1M}$, $l_{2M}$, $\sigma_M$, $\sigma_{\infty,M}$ is the strict transform of 
$l_1$, $l_2$, $\sigma$, 
$\sigma_{\infty}$, respectively. 
Then there exists an involution $\iota$ of $M$ over $\pr^1$ such that 
$\iota(\sigma_M)=\sigma_{\infty,M}$, $\iota(l_{1M})=\Gamma_{11}$ 
and $\iota(l_{2M})=\Gamma_{21}$ 
by Proposition \ref{invol_prop} (apply Proposition \ref{invol_prop} for the surface $M'$, 
where $M\to M'$ is obtained by contracting $\Gamma_{12}$ and $\Gamma_{22}$). 
Thus $\phi'\simeq\phi\circ\iota$ and $(X, E, \Delta)\simeq(X', E', \Delta')$. 
\end{proof}

\begin{table}[htp]
\caption{The list of the dual graphs of $E$ ($1/2<b/a<1$).}\label{maintable1}
\begin{tabular}{|c|c|}\hline
\textbf{$[$(3,2),1$]_0$} &
\begin{picture}(10,10)(0,0)
     \put(5, 3){\circle*{6}}
     \put(5, 3){\circle{10}}
\end{picture}
\\ \hline
\textbf{$[$(11,7),5$]_0$} &
\begin{picture}(80,10)(0,0)
     \put(5, 3){\circle*{6}}
     \put(5, 3){\circle{10}}
     \put(10, 3){\line(1, 0){8}}
     \put(23, 3){\circle*{10}}
     \put(28, 3){\line(1, 0){8}}
     \put(41, 3){\circle*{10}}
     \put(46, 3){\line(1, 0){8}}
     \put(59, 3){\circle*{10}}
     \put(64, 3){\line(1, 0){8}}
     \put(75, 3){\circle*{10}}
\end{picture}
\\ \hline
\textbf{$[$(5,3),3$]_0$(1)} &
\begin{picture}(80,9)(0,0)
     \put(5, 3){\circle*{6}}
     \put(5, 3){\circle{10}}
     \put(10, 3){\line(1, 0){8}}
     \put(23, 3){\circle*{10}}
     \put(28, 3){\line(1, 0){8}}
     \put(41, 3){\circle*{10}}
     \put(46, 3){\line(1, 0){8}}
     \put(59, 3){\circle*{10}}
     \put(64, 3){\line(1, 0){8}}
     \put(75, 3){\circle*{10}}
     \put(23, -1){\line(0, -1){8}}
     \put(23, -11){\circle*{10}}
\end{picture}
\\
 & 
\\ \hline
\textbf{$[$(5,3),3$]_0$(2)} &
\begin{picture}(80,10)(0,0)
     \put(5, 3){\circle*{10}}
     \put(10, 3){\line(1, 0){8}}
     \put(23, 3){\circle*{10}}
     \put(28, 3){\line(1, 0){8}}
     \put(41, 3){\circle*{6}}
     \put(41, 3){\circle{10}}
     \put(46, 3){\line(1, 0){8}}
     \put(59, 3){\circle*{10}}
     \put(64, 3){\line(1, 0){8}}
     \put(75, 3){\circle*{10}}
\end{picture}
\\ \hline
\textbf{$[$(9,5),7$]_{\times 43}$} &
\begin{picture}(85,10)(0,0)
     \put(5, 3){\circle*{6}}
     \put(5, 3){\circle{10}}
     \put(10, 3){\line(1, 0){8}}
     \put(23, 3){\circle*{10}}
     \put(28, 3){\line(1, 0){8}}
     \put(41, 3){\circle*{10}}
     \put(46, 3){\line(1, 0){8}}
     \put(59, 3){\circle*{10}}
    \put(70, 0){\makebox(0, 0)[b]{$\sqcup$}}
     \put(82, 3){\circle*{6}}
     \put(82, 3){\circle{10}}
\end{picture}
\\ \hline
\textbf{$[$(5,3),2;1,2$]_1$} &
\begin{picture}(25,10)(0,0)
     \put(5, 3){\circle*{6}}
     \put(5, 3){\circle{10}}
     \put(10, 3){\line(1, 0){8}}
     \put(23, 3){\circle*{10}}
\end{picture}
\\ \hline
\textbf{$[$(7,4),2;2,4$]_1$} &
\begin{picture}(80,10)(0,0)
     \put(5, 3){\circle*{10}}
     \put(10, 3){\line(1, 0){8}}
     \put(23, 3){\circle*{6}}
     \put(23, 3){\circle{10}}
     \put(28, 3){\line(1, 0){8}}
     \put(41, 3){\circle*{10}}
     \put(46, 3){\line(1, 0){8}}
     \put(59, 3){\circle*{10}}
     \put(64, 3){\line(1, 0){8}}
     \put(75, 3){\circle*{10}}
\end{picture}
\\ \hline
\textbf{$[$(13,7),2;5,10$]_1$} &
\begin{picture}(80,9)(0,0)
     \put(5, 3){\circle*{10}}
     \put(10, 3){\line(1, 0){8}}
     \put(23, 3){\circle*{6}}
     \put(23, 3){\circle{10}}
     \put(28, 3){\line(1, 0){8}}
     \put(41, 3){\circle*{10}}
     \put(46, 3){\line(1, 0){8}}
     \put(59, 3){\circle*{10}}
     \put(64, 3){\line(1, 0){8}}
     \put(76, 3){\circle*{10}}
     \put(41, -1){\line(0, -1){9}}
     \put(41, -12){\circle*{10}}
\end{picture}
\\
 & 
\\ \hline
\textbf{$[$(21,11),2;9,7$]_1$} &
\begin{picture}(70,10)(0,0)
     \put(5, 3){\circle*{6}}
     \put(5, 3){\circle{10}}
     \put(10, 3){\line(1, 0){8}}
     \put(23, 3){\circle*{10}}
     \put(27, 3){\line(1, 0){10}}
     \put(42, 3){\circle*{10}}
     \put(55, 0){\makebox(0, 0)[b]{$\sqcup$}}
     \put(68, 3){\circle*{6}}
     \put(68, 3){\circle{10}}
\end{picture}
\\ \hline
\textbf{$[$(2{\bi n}-1,{\bi n}+1),{\bi n};2({\bi n}-2),{\bi n}-2$]_1$} &
\begin{picture}(25,10)(0,0)
     \put(5, -6){\makebox(0, 0){\textcircled{\tiny $n$}}}
     \put(10, -4){\line(1, 0){10}}
     \put(23, -4){\circle*{10}}
\end{picture}
\\
\textbf{$[$(2{\bi m}+1,{\bi m}+1),3{\bi m}+2;2{\bi m},{\bi m}$]_1$} 
\,\scriptsize{$(n=3m+2)$}
&
\\ \hline
\textbf{$[$(4{\bi n}-3,2{\bi n}+1),{\bi n};4({\bi n}-2),3({\bi n}-2)$]_1$} &
\begin{picture}(65,10)(0,0)
     \put(5, -6){\makebox(0, 0){\textcircled{\tiny $n$}}}
     \put(10, -4){\line(1, 0){10}}
     \put(23, -4){\circle*{10}}
     \put(28, -4){\line(1, 0){8}}
     \put(41, -4){\circle*{10}}
     \put(46, -4){\line(1, 0){8}}
     \put(59, -4){\circle*{10}}     
\end{picture}
\\
\textbf{$[$(4{\bi m}+1,2{\bi m}+1),5{\bi m}+2;4{\bi m},3{\bi m}$]_1$} 
\,\scriptsize{$(n=5m+2)$} & 
\\ \hline
\textbf{$[$(2{\bi n}-2,{\bi n}),{\bi n};2({\bi n}-2),2({\bi n}-2)$]_{11}$} &
\begin{picture}(45,10)(0,0)
     \put(5, -4){\circle*{10}}
     \put(10, -4){\line(1, 0){8}}
     \put(23, -6){\makebox(0, 0){\textcircled{\tiny $n$}}}
     \put(28, -4){\line(1, 0){10}}
     \put(42, -4){\circle*{10}}
\end{picture}
\\
\textbf{$[$(2{\bi m}+1,{\bi m}+1),2{\bi m}+2;2{\bi m},2{\bi m}$]_{11}$} 
\,\scriptsize{$(n=2m+2)$} &
\\ \hline
\textbf{$[$(2{\bi n}-2,{\bi n}),{\bi n};2({\bi n}-2),2({\bi n}-2)$]_1$(1)} &
\begin{picture}(60,9)(0,0)
     \put(5, 1){\makebox(0, 0){\textcircled{\tiny $n$}}}
     \put(10, 3){\line(1, 0){8}}
     \put(23, 3){\circle*{10}}
     \put(28, 3){\line(1, 0){8}}
     \put(41, 3){\circle*{10}}
     \put(46, 3){\line(1, 0){8}}
     \put(59, 3){\circle*{10}}
     \put(41, -1){\line(0, -1){9}}
     \put(41, -12){\circle*{10}}
\end{picture}
\\
\textbf{$[$(2{\bi m}+1,{\bi m}+1),2{\bi m}+2;2{\bi m},2{\bi m}$]_1$(1)} 
\,\scriptsize{$(n=2m+2)$} & 
\\ \hline
\textbf{$[$(2{\bi n}-2,{\bi n}),{\bi n};2({\bi n}-2),2({\bi n}-2)$]_1$(2)} &
\begin{picture}(45,9)(0,0)
     \put(5, 1){\makebox(0, 0){\textcircled{\tiny $n$}}}
     \put(10, 3){\line(1, 0){8}}
     \put(23, 3){\circle*{10}}
     \put(28, 3){\line(1, 0){8}}
     \put(41, 3){\circle*{10}}
     \put(23, -1){\line(0, -1){9}}
     \put(23, -12){\circle*{10}}
\end{picture}
\\
\textbf{$[$(2{\bi m}+1,{\bi m}+1),2{\bi m}+2;2{\bi m},2{\bi m}$]_1$(2)} 
\,\scriptsize{$(n=2m+2)$} & 
\\ \hline
\textbf{$[$(4{\bi n}-5,2{\bi n}-1),{\bi n};4({\bi n}-2),5({\bi n}-2)$]_{32}$} &
\begin{picture}(80,10)(0,0)
     \put(5, -4){\circle*{10}}
     \put(10, -4){\line(1, 0){8}}
     \put(23, -6){\makebox(0, 0){\textcircled{\tiny $n$}}}
     \put(28, -4){\line(1, 0){8}}
     \put(40, -4){\circle*{10}}
     \put(45, -4){\line(1, 0){8}}
     \put(58, -4){\circle*{10}}
     \put(63, -4){\line(1, 0){8}}
     \put(76, -4){\circle*{10}}
\end{picture}
\\
\textbf{$[$(4{\bi m}+1,2{\bi m}+1),3{\bi m}+2;4{\bi m},5{\bi m}$]_{32}$} 
\,\scriptsize{$(n=3m+2)$} & 
\\ \hline
\textbf{$[$(7,5),3;4,5$]_1$} &
\begin{picture}(80,9)(0,0)
     \put(5, 3){\circle*{6}}
     \put(5, 3){\circle{10}}
     \put(10, 3){\line(1, 0){8}}
     \put(23, 3){\circle*{10}}
     \put(28, 3){\line(1, 0){8}}
     \put(41, 3){\circle*{10}}
     \put(46, 3){\line(1, 0){8}}
     \put(59, 3){\circle*{10}}
     \put(64, 3){\line(1, 0){8}}
     \put(75, 3){\circle*{10}}
     \put(41, -1){\line(0, -1){8}}
     \put(41, -11){\circle*{10}}
\end{picture}
\\
 & 
\\ \hline
\textbf{$[$(2{\bi n}-3,{\bi n}-1),{\bi n};2({\bi n}-2),3({\bi n}-2)$]_{111}$} &
\begin{picture}(45,9)(0,0)
     \put(5, 3){\circle*{10}}
     \put(10, 3){\line(1, 0){8}}
     \put(23, 1){\makebox(0, 0){\textcircled{\tiny $n$}}}
     \put(28, 3){\line(1, 0){8}}
     \put(41, 3){\circle*{10}}
     \put(23, -2){\line(0, -1){9}}
     \put(23, -12){\circle*{10}}
\end{picture}
\\
 & 
\\ \hline
\textbf{$[$(2{\bi n}-3,{\bi n}-1),{\bi n};2({\bi n}-2),3({\bi n}-2)$]_{21}$(1)} &
\begin{picture}(80,9)(0,0)
     \put(5, 3){\circle*{10}}
     \put(10, 3){\line(1, 0){8}}
     \put(23, 1){\makebox(0, 0){\textcircled{\tiny $n$}}}
     \put(28, 3){\line(1, 0){8}}
     \put(41, 3){\circle*{10}}
     \put(46, 3){\line(1, 0){8}}
     \put(59, 3){\circle*{10}}
     \put(64, 3){\line(1, 0){8}}
     \put(77, 3){\circle*{10}}
     \put(59, -1){\line(0, -1){9}}
     \put(59, -12){\circle*{10}}
\end{picture}
\\
& 
\\ \hline
\textbf{$[$(2{\bi n}-3,{\bi n}-1),{\bi n};2({\bi n}-2),3({\bi n}-2)$]_{21}$(2)} &
\begin{picture}(60,9)(0,0)
     \put(5, 3){\circle*{10}}
     \put(10, 3){\line(1, 0){8}}
     \put(23, 1){\makebox(0, 0){\textcircled{\tiny $n$}}}
     \put(28, 3){\line(1, 0){8}}
     \put(41, 3){\circle*{10}}
     \put(46, 3){\line(1, 0){8}}
     \put(59, 3){\circle*{10}}
     \put(41, -1){\line(0, -1){9}}
     \put(41, -12){\circle*{10}}
\end{picture}
\\
& 
\\ \hline
\textbf{$[$(4{\bi n}-6,2{\bi n}-2),{\bi n};4({\bi n}-2),6({\bi n}-2)$]_{11}$} &
\begin{picture}(115,9)(0,0)
     \put(5, -4){\circle*{10}}
     \put(10, -4){\line(1, 0){8}}
     \put(23, -4){\circle*{10}}
     \put(28, -4){\line(1, 0){8}}
     \put(41, -4){\circle*{10}}
     \put(46, -4){\line(1, 0){8}}
     \put(59, -6){\makebox(0, 0){\textcircled{\tiny $n$}}}
     \put(64, -4){\line(1, 0){8}}
     \put(75, -4){\circle*{10}}
     \put(80, -4){\line(1, 0){8}}
     \put(93, -4){\circle*{10}}
     \put(98, -4){\line(1, 0){8}}
     \put(111, -4){\circle*{10}}
\end{picture}
\\
\textbf{$[$(4{\bi m}+1,2{\bi m}+1),2{\bi m}+2;4{\bi m},6{\bi m}$]_{11}$} 
\,\scriptsize{$(n=2m+2)$} &
\\ \hline
\textbf{$[$(3,2),3;2,3$]_{1\infty}$} &
\begin{picture}(40,10)(0,0)
     \put(5, 3){\circle*{6}}
     \put(5, 3){\circle{10}}
     \put(18, 0){\makebox(0, 0)[b]{$\sqcup$}}
     \put(31, 3){\circle*{6}}
     \put(31, 3){\circle{10}}
\end{picture}
\\ \hline
\textbf{$[$(4{\bi n}-7,2{\bi n}-3),{\bi n};4({\bi n}-2),7({\bi n}-2)$]_{322}$} &
\begin{picture}(80,10)(0,0)
     \put(5, 3){\circle*{10}}
     \put(10, 3){\line(1, 0){8}}
     \put(23, 1){\makebox(0, 0){\textcircled{\tiny $n$}}}
     \put(28, 3){\line(1, 0){8}}
     \put(40, 3){\circle*{10}}
     \put(45, 3){\line(1, 0){8}}
     \put(58, 3){\circle*{10}}
     \put(63, 3){\line(1, 0){8}}
     \put(76, 3){\circle*{10}}
     \put(23, -2){\line(0, -1){9}}
     \put(23, -12){\circle*{10}}
\end{picture}
\\
 & 
\\ \hline
\textbf{$[$(4{\bi n}-7,2{\bi n}-3),{\bi n};4({\bi n}-2),7({\bi n}-2)$]_{43}$(1)} &
\begin{picture}(115,9)(0,0)
     \put(5, 3){\circle*{10}}
     \put(10, 3){\line(1, 0){8}}
     \put(23, 3){\circle*{10}}
     \put(28, 3){\line(1, 0){8}}
     \put(41, 3){\circle*{10}}
     \put(46, 3){\line(1, 0){8}}
     \put(59, 1){\makebox(0, 0){\textcircled{\tiny $n$}}}
     \put(64, 3){\line(1, 0){8}}
     \put(75, 3){\circle*{10}}
     \put(80, 3){\line(1, 0){8}}
     \put(93, 3){\circle*{10}}
     \put(98, 3){\line(1, 0){8}}
     \put(111, 3){\circle*{10}}
     \put(23, -1){\line(0, -1){9}}
     \put(23, -12){\circle*{10}}
\end{picture}
\\
&
\\ \hline
\textbf{$[$(4{\bi n}-7,2{\bi n}-3),{\bi n};4({\bi n}-2),7({\bi n}-2)$]_{43}$(2)} &
\begin{picture}(100,9)(0,0)
     \put(5, 3){\circle*{10}}
     \put(10, 3){\line(1, 0){8}}
     \put(23, 3){\circle*{10}}
     \put(28, 3){\line(1, 0){8}}
     \put(41, 1){\makebox(0, 0){\textcircled{\tiny $n$}}}
     \put(46, 3){\line(1, 0){8}}
     \put(59, 3){\circle*{10}}
     \put(64, 3){\line(1, 0){8}}
     \put(77, 3){\circle*{10}}
     \put(82, 3){\line(1, 0){8}}
     \put(95, 3){\circle*{10}}
     \put(23, -1){\line(0, -1){9}}
     \put(23, -12){\circle*{10}}
\end{picture}
\\
&
\\ \hline
\textbf{$[$(15,9),3;12,21$]_{5\infty 1}$} &
\begin{picture}(140,10)(0,0)
     \put(5, 3){\circle*{6}}
     \put(5, 3){\circle{10}}
     \put(10, 3){\line(1, 0){8}}
     \put(23, 3){\circle*{10}}
     \put(28, 3){\line(1, 0){8}}
     \put(41, 3){\circle*{10}}
     \put(46, 3){\line(1, 0){8}}
     \put(59, 3){\circle*{10}}
     \put(64, 3){\line(1, 0){8}}
     \put(75, 3){\circle*{10}}
     \put(80, 3){\line(1, 0){8}}
     \put(93, 3){\circle*{10}}
     \put(98, 3){\line(1, 0){8}}
     \put(111, 3){\circle*{10}}
     \put(124, 0){\makebox(0, 0)[b]{$\sqcup$}}
     \put(137, 3){\circle*{6}}
     \put(137, 3){\circle{10}}
\end{picture}
\\ \hline
\textbf{$[$(5,3),3;4,7$]_{2\infty 1}$} &
\begin{picture}(80,10)(0,0)
     \put(5, 3){\circle*{6}}
     \put(5, 3){\circle{10}}
     \put(10, 3){\line(1, 0){8}}
     \put(23, 3){\circle*{10}}
     \put(36, 0){\makebox(0, 0)[b]{$\sqcup$}}
     \put(49, 3){\circle*{6}}
     \put(49, 3){\circle{10}}
     \put(54, 3){\line(1, 0){8}}
     \put(67, 3){\circle*{10}}
\end{picture}
\\ \hline
\end{tabular}
\end{table}

\begin{table}[h]
\caption{The list of the dual graphs of $E$ ($b/a=1/2$).}\label{maintable2}
\begin{tabular}{|c|c|}\hline
\textbf{$[$(6,3),6$]_{\times 21}$} & 
\begin{picture}(145,9)(0,0)
     \put(5, 3){\circle*{10}}
     \put(10, 3){\line(1, 0){8}}
     \put(23, 3){\circle*{10}}
     \put(28, 3){\line(1, 0){8}}
     \put(41, 3){\circle*{10}}
     \put(46, 3){\line(1, 0){8}}
     \put(59, 3){\circle*{6}}
     \put(59, 3){\circle{10}}
     \put(64, 3){\line(1, 0){8}}
     \put(75, 3){\circle*{10}}
     \put(80, 3){\line(1, 0){8}}
     \put(93, 3){\circle*{10}}
     \put(98, 3){\line(1, 0){8}}
     \put(111, 3){\circle*{10}}
     \put(124, 0){\makebox(0, 0)[b]{$\sqcup$}}
     \put(137, 3){\circle*{6}}
     \put(137, 3){\circle{10}}
\end{picture}
\\ \hline
\textbf{$[$(6,3),3;6,12$]_{2\infty 11}$} & 
\begin{picture}(145,9)(0,0)
     \put(5, 3){\circle*{10}}
     \put(10, 3){\line(1, 0){8}}
     \put(23, 3){\circle*{10}}
     \put(28, 3){\line(1, 0){8}}
     \put(41, 3){\circle*{10}}
     \put(46, 3){\line(1, 0){8}}
     \put(59, 3){\circle*{6}}
     \put(59, 3){\circle{10}}
     \put(64, 3){\line(1, 0){8}}
     \put(75, 3){\circle*{10}}
     \put(80, 3){\line(1, 0){8}}
     \put(93, 3){\circle*{10}}
     \put(98, 3){\line(1, 0){8}}
     \put(111, 3){\circle*{10}}
     \put(124, 0){\makebox(0, 0)[b]{$\sqcup$}}
     \put(137, 3){\circle*{6}}
     \put(137, 3){\circle{10}}
\end{picture}
\\ \hline
\textbf{$[$(10,5),3;10,20$]_{4\infty 53}$} & 
\begin{picture}(160,9)(0,0)
     \put(5, 3){\circle*{10}}
     \put(10, 3){\line(1, 0){8}}
     \put(23, 3){\circle*{6}}
     \put(23, 3){\circle{10}}
     \put(28, 3){\line(1, 0){8}}
     \put(41, 3){\circle*{10}}
     \put(46, 3){\line(1, 0){8}}
     \put(59, 3){\circle*{10}}
     \put(64, 3){\line(1, 0){8}}
     \put(75, 3){\circle*{10}}
     \put(80, 3){\line(1, 0){8}}
     \put(93, 3){\circle*{10}}
     \put(98, 3){\line(1, 0){8}}
     \put(111, 3){\circle*{10}}
     \put(124, 0){\makebox(0, 0)[b]{$\sqcup$}}
     \put(137, 3){\circle*{6}}
     \put(137, 3){\circle{10}}
     \put(142, 3){\line(1, 0){8}}
     \put(155, 3){\circle*{10}}
\end{picture}
\\ \hline
\textbf{$[$(4,2),3;4,8$]_{2\infty 11}$} & 
\begin{picture}(105,9)(0,0)
     \put(5, 3){\circle*{10}}
     \put(10, 3){\line(1, 0){8}}
     \put(23, 3){\circle*{6}}
     \put(23, 3){\circle{10}}
     \put(28, 3){\line(1, 0){8}}
     \put(41, 3){\circle*{10}}
     \put(54, 0){\makebox(0, 0)[b]{$\sqcup$}}
     \put(67, 3){\circle*{10}}
     \put(72, 3){\line(1, 0){8}}
     \put(85, 3){\circle*{6}}
     \put(85, 3){\circle{10}}
     \put(90, 3){\line(1, 0){8}}
     \put(103, 3){\circle*{10}}
\end{picture}

\\ \hline
\end{tabular}
\end{table}

\subsection{Non-Gorenstein dual graphs}\label{graph_section}

As an immediate corollary of Theorems \ref{mainthm1} and \ref{mainthm2}, we can 
check the dual graph of $E$ for any normalized $(a, b)$-basic pair $(M, E)$. 
We note that the dual graph of $E$ is nothing but the dual graph of the minimal 
resolution of non-Gorenstein singular points on $S$, where $S$ is the log del Pezzo 
surface of the normalized multi-index $(a,b)$ corresponds to $(M, E)$.

\begin{corollary}
Let $a$, $b\in\Z_{>0}$ with $1/2\leq b/a<1$, $b>1$, 
$S$ be a log del Pezzo surface of the normalized multi-index $(a,b)$ and 
$(X, E, \Delta)$ be the associated normalized $(a, b)$-fundamental triplet. 
The dual graph of the minimal resolution of non-Gorenstein singular points on $S$ 
is determined by the type of $(X, E, \Delta)$ and 
are listed in Tables \ref{maintable1} and 
\ref{maintable2}. 
\end{corollary}

\subsection{Exceptional curves}\label{excep_curve_section}

\begin{proposition}[{cf.\ \cite[Lemma 4.13]{N}}]\label{excep_curve_prop}
Let $a$, $b$ be positive integers with $1/2\leq b/a<1$, $b>1$, 
$S$ be a log del Pezzo surface of the normalized multi-index $(a,b)$, 
$(M, E_M)$ be the associated normalized $(a, b)$-basic pair and 
$(X, E, \Delta)$ be the associated normalized $(a, b)$-fundamental triplet. 
An irreducible curve $C$ on $M$ is exceptional for $\alpha\colon M\to S$ 
if and only if one of the following conditions is satisfied: 
\begin{enumerate}
\renewcommand{\theenumi}{\arabic{enumi}}
\renewcommand{\labelenumi}{(\theenumi)}
\item\label{excep_curve_prop1}
$C$ is a $(-2)$-curve contracted by the elimination $\phi\colon M\to X$;
\item\label{excep_curve_prop2}
$C$ is the strict transform of an irreducible component of $E$. 
\end{enumerate}
\end{proposition}

\begin{proof}
The curve $C$ is $\alpha$-exceptional if and only if $(L_M\cdot C)=0$, where 
$L_M$ is the fundamental divisor of $(M, E_M)$. 
Set $C_X:=\phi_*C$. 
If $C$ is $\phi$-exceptional, then $(L_M\cdot C)=0$ if and only if $C$ is a 
$(-2)$-curve. 
Assume that $C$ is not $\phi$-exceptional 
and $C_X$ is not an irreducible component of $E$. 
By Claims \ref{mult_claim}, \ref{big_curve_claim} and \ref{small_possible_claim}, 
one of the following holds: 
\begin{enumerate}
\renewcommand{\theenumi}{\roman{enumi}}
\renewcommand{\labelenumi}{(\theenumi)}
\item\label{excep_curve_proof1}
$X=\pr^2$ and $C_X$ is a line.
\item\label{excep_curve_proof2}
$X=\F_n$ and $C_X$ is a fiber.
\item\label{excep_curve_proof3}
$X=\F_3$, $K_X+L$ is not big, $6\leq\deg\Delta_X\leq 8$
and $C_X$ is a section at infinity, where $L$ is the fundamental divisor.  
\end{enumerate}
In particular, $C_X$ is nonsingular. Thus $\deg(\Delta\cap C_X)=(L\cdot C_X)$ holds. 
However, by Corollary \ref{twocurve_cor}, the curve $C_X$ must be equal to 
an irreducible component of $E$, which leads to a contradiction. 
Therefore the assertion follows. 
\end{proof}

\end{document}